\documentclass[a4paper,11pt,reqno]{amsart}
\usepackage[utf8]{inputenc}
\usepackage{mathrsfs}
\usepackage{dsfont}
\usepackage{hyperref}
\usepackage{amsmath}
\usepackage{amssymb}
\usepackage{amsthm}
\usepackage{amsfonts}
\usepackage{amstext}
\usepackage{amsopn}
\usepackage{amsxtra}
\usepackage{mathrsfs}
\usepackage{dsfont}
\usepackage{esint}
\usepackage{enumitem}

\usepackage[backend=biber, style=alphabetic, natbib, labelnumber, defernumbers=true, language=english, autolang=other,maxbibnames=20,
doi=false,isbn=false,url=false,eprint=false, bibencoding=utf8, firstinits=true]{biblatex}

\AtBeginBibliography{\footnotesize}

\newtheorem{theorem}{Theorem}
\newtheorem{lemma}[theorem]{Lemma}
\newtheorem{proposition}[theorem]{Proposition}
\newtheorem{corollary}[theorem]{Corollary}
\newtheorem{remark}[theorem]{Remark}

\newcommand{\R}{\mathbb{R}}
\newcommand{\Z}{\mathbb{Z}}
\newcommand{\C}{\mathbb{C}}
\newcommand{\N}{\mathbb{N}}

\newcommand{\bS}{\mathbb{S}}
\newcommand{\bI}{\mathbb{I}}
\newcommand{\dps}{\displaystyle}
\newcommand{\ii}{\infty}
\newcommand\1{{\ensuremath {\mathds 1} }}
\renewcommand\phi{\varphi}

\newcommand{\wto}{\rightharpoonup}

\newcommand{\cU}{\mathcal{U}}

\newcommand{\cR}{\mathcal{R}}
\newcommand{\cM}{\mathcal{M}}

\newcommand{\cK}{\mathcal{K}}
\newcommand{\cE}{\mathcal{E}}
\newcommand{\cF}{\mathcal{F}}
\newcommand{\cN}{\mathcal{N}}
\newcommand{\cD}{\mathcal{D}}
\newcommand{\cG}{\mathcal{G}}

\newcommand{\cH}{\mathcal{H}}

\newcommand\pscal[1]{{\ensuremath{\left\langle #1 \right\rangle}}}
\newcommand{\rd}{{\rm d}}
\newcommand{\Vol}{{\rm Vol}}
\newcommand{\vol}{{\rm vol}}
\newcommand{\norm}[1]{ \left\| #1 \right\|}
\newcommand{\tr}{{\rm Tr}\,}
\newcommand{\Tr}{{\rm Tr}\,}

\renewcommand{\geq}{\geqslant}
\renewcommand{\leq}{\leqslant}

\renewcommand{\tilde}{\widetilde}
\newcommand{\eps}{\varepsilon}
\usepackage{color}

\newcommand{\nn}{\nonumber}

\title[Finite-rank Lieb-Thirring inequality]{Optimizers for the finite-rank Lieb-Thirring inequality}

\author{Rupert L. Frank}
\address[Rupert L. Frank]{Mathematics 253-37, Caltech, Pasa\-de\-na, CA 91125, USA, and Mathematisches Institut, Ludwig-Maximilans Universit\"at M\"unchen, Theresienstr. 39, 80333 M\"unchen, Germany, and Munich Center for Quantum Science and Technology (MCQST), Schellingstr. 4, 80799 M\"unchen, Germany}
\email{rlfrank@caltech.edu}

\author{David Gontier}
\address[David Gontier]{CEREMADE, University of Paris-Dauphine, PSL University, 75016 Paris, France}
\email{gontier@ceremade.dauphine.fr}

\author{Mathieu Lewin}
\address[Mathieu Lewin]{CNRS and CEREMADE, University of Paris-Dauphine, PSL University, 75016 Paris, France}
\email{mathieu.lewin@math.cnrs.fr}

\date{\today}

\begin{document}

 \begin{abstract}
The finite-rank Lieb-Thirring inequality provides an estimate on a Riesz sum of the $N$ lowest eigenvalues of a Schr\"odinger operator $-\Delta-V(x)$ in terms of an $L^p(\mathbb{R}^d)$ norm of the potential $V$. We prove here the existence of an optimizing potential for each $N$, discuss its qualitative properties and the Euler--Lagrange equation (which is a system of coupled nonlinear Schr\"odinger equations) and study in detail the behavior of optimizing sequences. In particular, under the condition $\gamma>\max\{0,2-d/2\}$ on the Riesz exponent in the inequality, we prove the compactness of all the optimizing sequences up to translations. We also show that the optimal Lieb-Thirring constant cannot be stationary in $N$, which sheds a new light on a conjecture of Lieb-Thirring. In dimension $d=1$ at $\gamma=3/2$, we show that the optimizers with $N$ negative eigenvalues are exactly the Korteweg-de Vries $N$--solitons and that optimizing sequences must approach the corresponding manifold. Our work also covers the critical case $\gamma=0$ in dimension $d\geq3$ (Cwikel-Lieb-Rozenblum inequality) for which we exhibit and use a link with invariants of the Yamabe problem.

\bigskip

\noindent \sl \copyright~2021 by the authors. This paper may be reproduced, in its entirety, for non-commercial purposes.
 \end{abstract}

 \maketitle

 \tableofcontents

\section{Introduction}

Lieb-Thirring inequalities provide estimates on the Riesz sums of the negative eigenvalues of a Schrödinger operator $-\Delta-V(x)$ in terms of the size of the potential $V(x)$ in some $L^p$ spaces. They play a crucial role in the rigorous study of large fermionic quantum systems, and were introduced by Lieb and Thirring in 1975~\cite{LieThi-75,LieThi-76} to provide a short and elegant proof of the stability of matter, previously shown by Dyson and Lenard~\cite{DysLen-67,DysLen-68}. Many interesting and challenging mathematical questions are still open today regarding these inequalities~\cite{Frank-23,FraLapWei-LT}. One important problem is to determine the value of the best Lieb-Thirring constants, which has some implications in Density Functional Theory~\cite{LewLieSei-19_ppt}. In two recent works~\cite{FraGonLew-21,FraGonLew-21b} we have suggested a new scenario for these best constants, based on the study of the `finite rank' inequalities where only the $N$ first eigenvalues are considered. We pursue this analysis in this paper, with an approach different from that of~\cite{FraGonLew-21} allowing to reach the complete range of possible exponents.

We start by recalling what the Lieb-Thirring inequalities are. Let $d\geq1$ be the space dimension and
\begin{equation}
\gamma\begin{cases}
\geq\frac12&\text{for $d=1$,}\\
>0&\text{for $d=2$,}\\
\geq0&\text{for $d\geq3$.}
\end{cases}
\label{eq:gamma}
\end{equation}
Let $V\in L^1_{\rm loc}(\R^d)$ be a real-valued function with its positive part $V_+=\max(0,V)$ in the Lebesgue space $L^{\gamma+d/2}(\R^d)$. The operator $-\Delta-V$ has a well defined (Friedrichs) self-adjoint realization defined through its quadratic form and its essential spectrum is contained in the half line $[0,\ii)$~\cite{ReeSim2,ReeSim4}. We denote by $\lambda_j(-\Delta-V)$ the $j$th min-max level which equals the $j$th negative eigenvalue counted with multiplicity if it exists and vanishes otherwise. The Lieb-Thirring inequality reads
\begin{equation}
\sum_{j=1}^{\ii}|\lambda_j(-\Delta-V)|^\gamma\leq L_{\gamma,d}\int_{\R^d}V(x)_+^{\gamma+\frac{d}2}\,\rd x,
\label{eq:LT}
\end{equation}
where $L_{\gamma,d}$ is by definition the best (lowest) constant in the inequality.
Note that the left side of~\eqref{eq:LT} increases if we replace $V$ by $V_+$. Hence, without loss of generality, we may always assume that $V\geq0$, which justifies our choice of putting a minus sign in the operator $-\Delta-V$. At $\gamma=0$ the left side of~\eqref{eq:LT} is by definition taken equal to the number of negative eigenvalues. The Lieb-Thirring theorem~\cite{LieThi-75,LieThi-76} states that
$$L_{\gamma,d}<\ii$$
for $\gamma$ and $d$ as in~\eqref{eq:gamma} (see also~\cite{LieSei-09,Frank-23,FraLapWei-LT}). The end point cases $\gamma=0$ in $d\geq3$ and $\gamma=1/2$ in $d=1$ are respectively due to  Cwikel-Lieb-Rozenblum~\cite{Cwikel-77,Lieb-76b,Rozenbljum-72} and Weidl~\cite{Weidl-96}.

As we have mentioned above, it is important for applications to determine the value of $L_{\gamma,d}$. So far this is only known for $\gamma\geq3/2$ in all dimensions~\cite{LieThi-76,AizLie-78,LapWei-00,BenLos-00} and $\gamma=1/2$ in dimension $d=1$~\cite{HunLieTho-98}. The best currently known bound in the physically most relevant case $\gamma=1$ is in~\cite{FraHunJexNam-21}. A natural question is whether there exists an optimal potential and whether $L_{\gamma,d}$ could be given by a potential having finitely many eigenvalues. This leads us to introducing what we call the \emph{finite-rank Lieb-Thirring inequality}
\begin{equation}
\boxed{\sum_{j=1}^{N}|\lambda_j(-\Delta-V)|^\gamma\leq L^{(N)}_{\gamma,d}\int_{\R^d}V(x)_+^{\gamma+\frac{d}2}\,\rd x,}
\label{eq:LT_N}
\end{equation}
for all $N\in\N$. Again we denote by
\begin{equation}
\boxed{L^{(N)}_{\gamma,d}=\sup_{\substack{0\leq V\in L^{\gamma+d/2}(\R^d)\\ V\neq0}}\frac{\sum_{j=1}^{N}|\lambda_j(-\Delta-V)|^\gamma}{\int_{\R^d}V(x)^{\gamma+d/2}\,\rd x}}
 \label{eq:variational_L}
\end{equation}
the best constant in the inequality~\eqref{eq:LT_N}. At $\gamma=0$, our convention is that the left side of~\eqref{eq:LT_N} equals $\min(N,\#\{j\;:\;\lambda_j(-\Delta-V)<0\})$, that is, the number of negative eigenvalues truncated to $N$. The constants $L^{(N)}_{\gamma,d}$ form a non-decreasing sequence which converges to $L_{\gamma,d}$ in the limit $N\to\ii$. Our goal in this paper is to prove the existence of an optimal potential $V$ for $L^{(N)}_{\gamma,d}$ and to study whether it could possibly be an optimizer for $L_{\gamma,d}$.

The case $N=1$ is studied at length in~\cite{LieThi-76}. A duality argument (also recalled in~\cite{Frank-23,FraGonLew-21}) implies that
$$L^{(1)}_{\gamma,d}=\left(\frac{2\gamma}{2\gamma+d}\right)^{\gamma+\frac{d}2}\left(\frac{d}{2\gamma}\right)^{\frac{d}2}\left(C_{p,d}^\text{GNS}\right)^{\frac{d}2},\qquad p=\frac{2d+4\gamma}{d-2+2\gamma},$$
where $C_{p,d}^\text{GNS}$ is the best constant in the Gagliardo-Nirenberg-Sobolev inequality
\begin{equation}
 \left(\int_{\R^d}|u(x)|^p\,\rd x\right)^{\frac{4}{d(p-2)}}\leq C_{p,d}^\text{GNS}\left(\int_{\R^d}|u(x)|^2\,\rd x\right)^{\frac{(2-d)p+2d}{d(p-2)}}\int_{\R^d}|\nabla u(x)|^2\,\rd x.
 \label{eq:GNS}
\end{equation}
The condition~\eqref{eq:gamma} on $\gamma$ corresponds to the usual Sobolev constraint $2\leq p\leq 2^*$. At $\gamma=0$ the exponent of the $L^2$ norm vanishes and we obtain the Sobolev inequality. There are always optimizers for $L^{(1)}_{\gamma,d}$ when $\gamma$ is as in~\eqref{eq:gamma} (excluding $\gamma=1/2$ if $d=1$). Those can be expressed as
$$V=\beta^2Q(\beta x-X)^{p-2},\qquad \beta>0,\ X\in \R^d.$$
When $\gamma>0$, $Q$ is the unique radial positive nonlinear Schrödinger ground state, that is, the optimizer of~\eqref{eq:GNS} which solves the nonlinear equation
$$-\Delta Q-Q^{p-1}+Q=0.$$
At $\gamma=0$, $Q$ solves instead the Emden-Fowler equation
$$-\Delta Q=Q^{\frac{d+2}{d-2}}.$$
It is known~\cite{HunLieTho-98,GarGreeKruMiu-74,LieThi-76} that $L_{\gamma,d}=L^{(1)}_{\gamma,d}$ for $d=1$ and $\gamma\in\{\frac12,\frac32\}$ and numerical simulations~\cite{LieThi-76,Levitt-14,FraGonLew-21} suggest the same for $\gamma\in(\frac12,\frac32)$.

When $N\geq2$, the bound $0\geq \lambda_j(-\Delta-V)\geq\lambda_1(-\Delta-V)$ implies that $L^{(N)}_{\gamma,d}\leq NL^{(1)}_{\gamma,d}<\ii$ but the behavior in $N$ is not optimal since $L^{(N)}_{\gamma,d}$ converges to $L_{\gamma,d}$ in the limit $N\to\ii$. In this paper we prove the following results:
\begin{itemize}[leftmargin=* ,parsep=0cm,itemsep=0cm,topsep=0cm]
 \item \textsl{(Existence)} $L^{(N)}_{\gamma,d}$ \textbf{admits an optimizer} for all $N\in\N$ and all $\gamma$ as in~\eqref{eq:gamma}. The case $\gamma=1/2$ in dimension $d=1$ is studied in~\cite{HunLieTho-98} and only the other cases are considered here. In Theorems~\ref{thm:bubble} and~\ref{thm:bubble_critical} below, we show a \textbf{bubble decomposition} for the variational principle~\eqref{eq:variational_L} which provides a rather precise description of the behavior of all the maximizing sequences, including the critical case $\gamma=0$ in dimensions $d\geq3$. The existence of optimizers follows immediately from this decomposition.

 \smallskip

 \item \textsl{(Properties of optimizers)} In Theorems~\ref{thm:equation} and~\ref{thm:equation_critical} we prove that any optimizer $V$ of $L^{(N)}_{\gamma,d}$ solves a particular \textbf{nonlinear equation}, and use it to derive its regularity and decay at infinity.

  \smallskip

 \item \textsl{(Compactness and binding for $\gamma>2-d/2$)} We then prove in Theorem~\ref{thm:existence} that when
 \begin{equation}
\gamma>\max\left(0,2-\frac{d}2\right),
 \label{eq:gamma_intro}
 \end{equation}
 \textbf{all the (properly normalized) maximizing sequences for~\eqref{eq:variational_L} are compact, up to space translations} in $L^{\gamma+d/2}(\R^d)$, and thus converge to an optimizer after extraction of a subsequence. In addition we prove that
 \begin{equation}
  L_{\gamma,d}>L^{(N)}_{\gamma,d}\qquad\text{for all $N\in\N$.}
  \label{eq:increasing}
 \end{equation}
The best Lieb-Thirring constant can thus \textbf{never be attained for a potential having finitely many bound states}. The proof is inspired by~\cite{GonLewNaz-21,FraGonLew-21} and goes by showing that
\begin{equation}
  L^{(2N)}_{\gamma,d}>L^{(N)}_{\gamma,d}\qquad\text{for all $N\in\N$.}
  \label{eq:binding_intro}
 \end{equation}
In particular, we have $L^{(2)}_{\gamma,d}>L^{(1)}_{\gamma,d}$. The idea is to compute the exponentially small nonlinear correction due to quantum tunnelling when two optimizers for $L^{(N)}_{\gamma,d}$ are placed far apart. It is negative because of~\eqref{eq:gamma_intro}. Moreover we will see below that the constraint~\eqref{eq:gamma_intro} is necessary at least in dimensions $d=1$ and $d\geq 4$.

 \smallskip

\item \textsl{(The integrable case)} The constraint~\eqref{eq:gamma_intro} on $\gamma$ is optimal in dimension $d=1$. The result cannot hold at $\gamma=3/2$ since, as we have recalled, $L^{(N)}_{3/2,1}=L^{(1)}_{3/2,1}=L_{3/2,1}$ for all $N\in\N$~\cite{HunLieTho-98,GarGreeKruMiu-74,LieThi-76}. We prove that
there exist \textbf{non-compact maximizing sequences}. More precisely, the optimizers for $L^{(N)}_{3/2,1}$ are exactly \textbf{the KdV $n$-solitons} for $n\leq N$, which form a non-compact manifold. We also show that any maximizing sequence approaches the manifold of KdV solitons in $L^2$.

 \smallskip

\item \textsl{(The critical case)} The constraint~\eqref{eq:gamma_intro} on $\gamma$ is also optimal in dimensions $d\geq4$. At $\gamma=0$, using a link with the Yamabe problem and ideas from~\cite{AmmHum-06,GlaGroMar-78} we prove in Corollary~\ref{cor:L2=L1} that
$$L^{(2)}_{0,d}=L^{(1)}_{0,d}\qquad\text{in all dimensions $d\geq3$}$$
which should be compared with the strict inequality~\eqref{eq:binding_intro} for $\gamma>0$ and $N=1$. We can also show that
$$L^{(d+2)}_{0,d}>L^{(1)}_{0,d}\qquad\text{in all dimensions $d\geq7$}.$$
\end{itemize}

\medskip

\noindent The strict inequalities~\eqref{eq:increasing} and~\eqref{eq:binding_intro} were proved in~\cite{FraGonLew-21} for $\gamma\geq1$, using a duality principle for $L^{(N)}_{\gamma,d}$. Since no such principle is at our disposal for $\gamma<1$, we develop in this paper a completely new approach, which also sheds a new light on the previously known cases and allows to reach the critical case $\gamma=0$ in dimensions $d\geq3$.

Our work leaves open the question of what $L_{\gamma,d}$ could be for $\max(0,2-d/2)<\gamma<3/2$ in dimensions $d\geq2$, but at least says that the finite-rank case cannot be optimal. Based on a statistical mechanics interpretation of the Lieb-Thirring inequality, we proposed in~\cite{FraGonLew-21b} that a sequence $V^{(N)}$ of optimizers for $L^{(N)}_{\gamma,d}$ could converge to a (not necessarily constant) periodic potential. This was supported by an exact computation at $\gamma=3/2$ in $d=1$ and numerical simulations in dimension $d=2$.

In the next two sections we state all our main results for the subcritical case ($\gamma>1/2$ in dimension $d=1$ and $\gamma>0$ in $d\geq2$) and the critical case ($\gamma=0$ in $d\geq3$), respectively. At $\gamma=0$ we in fact study another constant $\ell^{(N)}_{0,d}$ related to $L^{(N)}_{0,d}$, which was previously considered in the context of the Yamabe problem in~\cite{AmmHum-06}. The rest of the paper contains the proof of all our results.

\section{Main results in the subcritical case}
In this section we state all our results in the case that $\gamma>1/2$ in dimension $d=1$ and $\gamma>0$ for $d\geq2$. For the reader's convenience, we prefer to discuss the critical case $\gamma=0$ in dimension $d\geq3$ separately in Section~\ref{sec:results_critical}.

\subsection{Bubble decomposition}
We first clarify the notion of optimizers. In the whole paper we say that a non-trivial potential $V\geq0$ is a \textbf{maximizer (or optimizer) for $L^{(N)}_{\gamma,d}$} when it realizes the supremum on the right of~\eqref{eq:variational_L}. Note that a maximizer does not necessarily have exactly $N$ negative eigenvalues. It could have more or less. If we rescale a potential $V$ in the manner $t^2V(t\cdot)$ we see that the functional in the supremum in~\eqref{eq:variational_L} is invariant, whereas the integral becomes
$$t^{2\gamma+d}\int_{\R^d}V(tx)^{\gamma+\frac{d}2}\,\rd x=t^{2\gamma}\int_{\R^d}V(x)^{\gamma+\frac{d}2}\,\rd x.$$
Therefore, choosing $t$ appropriately we may also freely assume that $V$ is normalized in $L^{\gamma+d/2}(\R^d)$:
\begin{equation}
 L^{(N)}_{\gamma,d}=\sup_{\substack{0\leq V\in L^{\gamma+d/2}(\R^d)\\ \int_{\R^d}V^{\gamma+d/2}=1}}\sum_{j=1}^{N}|\lambda_j(-\Delta-V)|^\gamma.
 \label{eq:variational_L_normalized}
\end{equation}
We call a solution of this supremum a \textbf{normalized maximizer (or optimizer)}. Similarly, maximizing sequences for the supremum in~\eqref{eq:variational_L_normalized} will be called \textbf{normalized maximizing sequences for $L^{(N)}_{\gamma,d}$}.

Our first main result concerns the behavior of an arbitrary maximizing sequence $(V_n)$ for the variational problem~\eqref{eq:variational_L_normalized}.

\begin{theorem}[Bubble decomposition for $L^{(N)}_{\gamma,d}$ -- subcritical case]\label{thm:bubble}
Let $\gamma>0$ if $d\geq2$ and $\gamma>1/2$ if $d=1$. For $N\in\N$, consider a normalized maximizing sequence of non-negative potentials $(V_n)$ for $L^{(N)}_{\gamma,d}$, that is,
\begin{equation}
 \int_{\R^d}V_n(x)^{\gamma+d/2}\,\rd x=1,\qquad \lim_{n\to\ii}\sum_{j=1}^N|\lambda_j(-\Delta-V_n)|^\gamma=L^{(N)}_{\gamma,d}.
 \label{eq:choice_V_n}
\end{equation}
Then there exist

\smallskip

\begin{itemize}[leftmargin=* ,parsep=0cm,itemsep=0cm,topsep=0cm]
 \item integers $N_1,...,N_K\geq1$ (for some $K\geq1$) so that
\begin{equation}
\sum_{k=1}^KN_k=N\quad\text{and}\quad L^{(N_k)}_{\gamma,d}= L^{(N)}_{\gamma,d}\quad \text{for all $k=1,...,K$;}
\label{eq:decomp_N}
\end{equation}

\smallskip

\item non-negative functions $0\neq V^{(k)}\in L^{\gamma+d/2}(\R^d,\R_+)$ for $k=1,...,K$, which are maximizers for $L^{(N_k)}_{\gamma,d}$ and satisfy $\lambda_{N_k+1}(-\Delta-V^{(k)})=0$ if $N_k<N$, hence are also maximizers for $L^{(N)}_{\gamma,d}$, for all $k=1,...,K$;

\smallskip

\item sequences $x_n^{(k)}\in\R^d$ with
$$\lim_{n\to\ii}|x^{(k)}_n-x^{(k')}_n|=+\ii\qquad\text{for all $1\leq k\neq k'\leq K$;}$$
\end{itemize}

\smallskip

\noindent such that, after extracting a (not displayed) subsequence, we have
\begin{equation}
\boxed{\lim_{n\to\ii}\norm{V_n-\sum_{k=1}^K V^{(k)}(\cdot -x_n^{(k)})}_{L^{\gamma+d/2}(\R^d)}=0.}
\label{eq:bubble_v}
\end{equation}
\end{theorem}

Since $N\mapsto L^{(N)}_{\gamma,d}$ is non-decreasing, the equality~\eqref{eq:decomp_N} implies that
$$L^{(N)}_{\gamma,d}=L^{(m)}_{\gamma,d}\qquad\text{for all $\min_{1\leq k\leq K}N_k\leq m\leq N$.}$$

The functions $V^{(k)}$ are sometimes called `bubbles'. In the case $K=1$ we must have $N_1=N$; the sequence $(V_n)$ is precompact in $L^{\gamma+d/2}(\R^d)$  up to translations, and thus converges to a Lieb-Thirring optimizer for $L^{(N)}_{\gamma,d}$. When $K\geq2$, the limit~\eqref{eq:bubble_v} means that only the group of space translations can explain the non-compactness of a normalized maximizing sequence. Namely, the only possibility is that a sequence $V_n$ splits into $K$ independent `bubbles' $V^{(k)}$ placed very far apart. Each bubble generates at most $N_k<N$ eigenvalues and maximizes the problem $L^{(N_k)}_{\gamma,d}$. The spectrum of $-\Delta-V_n$ is asymptotically given by the union of the spectra of the $K$ operators $-\Delta-V^{(k)}$, counting multiplicities.

There are many results of the same kind in the literature on variational problems~\cite{Struwe-84,BreCor-85,Lions-87,Gerard-98,HmiKer-05}. The difficulty here is that the functional to be optimized in~\eqref{eq:variational_L} is highly nonlinear, since the eigenvalues depend on the potential in a very indirect and nonlocal way. In addition, the potential  only lives in $L^{\gamma+d/2}(\R^d)$ and no additional regularity is known \emph{a priori}. Our proof follows the usual strategy of extracting bubbles one after another using ideas from Lions' concentration-compactness method~\cite{Lions-84,Lewin-conc-comp}. The main difficulties are to show that our problem is essentially local (\emph{i.e.}, two arbitrary potentials far away generate essentially independent eigenvalues) and locally compact (the eigenvalues converge for a sequence $V_n$ of fixed compact support converging to some $V$ weakly but not necessarily strongly). The sub-criticality of our problem is better seen in terms of the corresponding eigenfunctions which live in $H^1(\R^d)$. The detailed proof is provided later in Section~\ref{sec:proof_bubbles} and it uses some inspiration from~\cite{EstLewSer-21b}. Our method could be useful to deal with other problems of spectral optimization theory.

Since any bubble $V^{(k)}$ is a maximizer for $L_{\gamma,d}^{(N)}$, Theorem~\ref{thm:bubble} immediately implies that there are Lieb-Thirring optimizers for all $N\geq1$.

\begin{corollary}[Existence of optimizers -- subcritical case]\label{cor:existence}
Let $\gamma>0$ if $d\geq2$ and $\gamma>1/2$ if $d=1$. Then $L^{(N)}_{\gamma,d}$ admits an optimizer for all $N\geq1$.
\end{corollary}

This corollary extends our previous existence result~\cite[Thm.~6]{FraGonLew-21} which only covered the case $\gamma\geq1$ since it was based on a	 duality principle only known under this additional constraint. The case $\gamma=1/2$ in dimension $d=1$ has been fully solved in~\cite{HunLieTho-98}. It is proved there that $L^{(N)}_{1/2,1}=L^{(1)}_{1/2,1}$ for all $N\geq1$ and that $V=c\delta_y$ (Dirac delta at a point $y$) are the only optimizers for any $c>0$ (allowing potentials being finite Borel measures). These potentials have the unique eigenvalue $-c^2/4$. The case $\gamma=0$ in dimensions $d\geq3$ (Cwikel-Lieb-Rozenblum inequality) will be considered in Section~\ref{sec:results_critical}.

\subsection{Euler-Lagrange equation}

Next, we show that any maximizer, when it exists, solves a nonlinear equation and we use this property to infer its regularity and decay at infinity.

\begin{theorem}[Euler-Lagrange equation -- subcritical case]\label{thm:equation}
Let $\gamma>0$ if $d\geq2$ and $\gamma>1/2$ if $d=1$. Let $N\in\N$ and assume that $V_*\geq0$ is a normalized optimizer for $L^{(N)}_{\gamma,d}$. Denote by $\lambda_j:=\lambda_j(-\Delta-V_*)$ the corresponding min-max values.
The operator $-\Delta-V_*$ has finitely many negative eigenvalues, that is, $\lambda_{j}=0$ for some large enough $j$. Let $M$ be the number of negative eigenvalues repeated according to their multiplicity, that is, the smallest integer such that $\lambda_{M+1}=0$. We have
\begin{equation}
V_*(x)=\left(\frac{2\gamma}{(d+2\gamma)L_{\gamma,d}^{(N)}}\sum_{j=1}^{\min(N,M)}|\lambda_j|^{\gamma-1}|u_j(x)|^2\right)^{\frac1{\gamma+\frac{d}2-1}},
 \label{eq:equation_V_optimal}
\end{equation}
where $u_j$ is any associated orthonormal system of eigenfunctions. In the case that $\lambda_{N+1}<0$, that is $M>N$, we have the spectral gap
\begin{equation}
\lambda_N< \lambda_{N+1}.
\label{eq:no-unfilled-shell}
\end{equation}
Finally, in any case $V_*$ is real-analytic and tends to 0 exponentially fast at infinity.
\end{theorem}

The proof is provided in Section~\ref{sec:proof_equation}. The inequality~\eqref{eq:no-unfilled-shell} means that there are ``no unfilled shells for Lieb-Thirring maximizers'' in the spirit of what is known in Hartree-Fock theory~\cite{BacLieLosSol-94}. The last filled eigenvalue $\lambda_{N'}$ with $N':=\min(N,M)$ is always strictly less than $\lambda_{N'+1}$. The latter can either be an eigenvalue or be equal to $0$ (the bottom of the essential spectrum).

The formula~\eqref{eq:equation_V_optimal} means that the corresponding eigenfunctions $u_j$ solve a system of $N'$ coupled nonlinear Schr\"odinger equations in the form
\begin{equation}
 \left(-\Delta-c_{\gamma,d}^{(N)}\left(\sum_{j=1}^{N'}|\lambda_j|^{\gamma-1}|u_j|^2\right)^{\frac1{\gamma+\frac{d}2-1}}\right)u_j=\lambda_j\,u_j
 \label{eq:NLS}
\end{equation}
with the constant
\begin{equation}
c_{\gamma,d}^{(N)}:=\left( \frac{2\gamma}{(d+2\gamma)L_{\gamma,d}^{(N)}}\right)^{\frac1{\gamma+\frac{d}2-1}}.
\label{eq:c_gamma_d}
\end{equation}

\subsection{Application to $\gamma>\max(0,2-d/2)$}

In this subsection we consider the case $\gamma>\max(0,2-d/2)$, where we can show that $N\mapsto L^{(N)}_{\gamma,d}$ cannot be constant for large $N$.

\begin{theorem}[Existence for $\gamma>\max(0,2-d/2)$]\label{thm:existence}
Assume that
\begin{equation}
\gamma>\max\left(0,2-\frac{d}2\right).
\label{eq:cond_gamma}
\end{equation}
Then for every $N\geq1$,

\smallskip

\begin{itemize}[leftmargin=* ,parsep=0cm,itemsep=0cm,topsep=0cm]
 \item up to translations, all the normalized maximizing sequences $(V_n)$ for $L^{(N)}_{\gamma,d}$ are precompact and converge to a maximizer after extraction of a subsequence;

 \smallskip

 \item $L^{(2N)}_{\gamma,d}>L^{(N)}_{\gamma,d}$, hence in particular
\end{itemize}
\begin{equation}
\boxed{L_{\gamma,d}>L^{(N)}_{\gamma,d},\qquad\text{for all $N\geq1$.}}
\label{eq:not_finite}
\end{equation}
\end{theorem}

The inequality~\eqref{eq:not_finite} was shown in~\cite{FraGonLew-21} under the condition that $\gamma\geq1$. The compactness of all the normalized maximizing sequences and the validity of the strict inequality $L^{(2N)}_{\gamma,d}>L^{(N)}_{\gamma,d}$ for all $N\geq1$ are new, even for $\gamma\geq1$.

The inequality~\eqref{eq:not_finite} means that the finite-rank Lieb-Thirring constant $L^{(N)}_{\gamma,d}$ is never optimal for the unconstrained Lieb-Thirring problem $L_{\gamma,d}$. We note, however, that this does not prevent the existence of an optimal potential for $L_{\gamma,d}$ that would have infinitely many eigenvalues. We refer to~\cite{FraGonLew-21b} for a possible scenario concerning the behavior of a sequence $V^{(N)}$ of optimal potentials for $L^{(N)}_{\gamma,d}$ in the limit $N\to\ii$ and for the implications with regard to the Lieb-Thirring conjecture.

We will see later in Corollary~\ref{cor:L2=L1} that $L^{(2)}_{0,d}=L^{(1)}_{0,d}$  in dimensions $d\geq3$ at $\gamma=0$. Hence the constraint~\eqref{eq:cond_gamma} on $\gamma$ is optimal in dimensions $d\geq4$. It is also optimal in dimension $d=1$. In fact, at $\gamma=3/2$ all the $L^{(N)}_{3/2,1}$ are equal to each other,
\begin{equation}
L^{(N)}_{3/2,1}=L_{3/2,1}=\frac3{16},\qquad\forall N\in\N,
\label{eq:gamma3/2}
\end{equation}
as was proved in~\cite{GarGreeKruMiu-74,LieThi-76,BenLos-00} and is discussed in the next subsection. There are optimizers with exactly $N$ negative eigenvalues for every $N$, which are therefore not maximizers for $L^{(m)}_{\gamma,d}$ for $m<N$. Finally, there are also non-compact normalized maximizing sequences up to translations for $L^{(N)}_{\gamma,d}$ for every $N\geq2$. For $\gamma<3/2$ numerical simulations~\cite{LieThi-76,Levitt-14,FraGonLew-21} seem to indicate that $L^{(N)}_{\gamma,1}=L^{(1)}_{\gamma,1}$ but with all optimizers having only one negative eigenvalue, as it is the case for $\gamma=1/2$~\cite{HunLieTho-98}. We do not know if the condition~\eqref{eq:cond_gamma} is optimal in dimensions $d\in\{2,3\}$.

\subsection{Application to $\gamma=3/2$ in $d=1$}
In order to describe in detail what maximizing sequences are doing at $\gamma=3/2$ in one dimension, we introduce Korteweg-de Vries (KdV) $N$-solitons. Following for instance~\cite{KayMos-56,KilVis-20_ppt}, for $\vec\beta=(\beta_1,...,\beta_N)$ with $\beta_1>\beta_2>\cdots>\beta_N>0$ and $\vec{X}=(X_1,...,X_N)\in\R^N$, we introduce the $N\times N$ matrix
$$A_{jk}(x):=\delta_{jk}+\frac{1}{\beta_j+\beta_k}e^{-\beta_j(x-X_j)-\beta_k(x-X_k)}$$
and the function
\begin{equation}
Q_{\vec{\beta},\vec{X}}(x):=-2\frac{d^2}{\rd x^2}\log\det(A(x)).
\label{eq:KdV-N-soliton}
\end{equation}
Then $Q_{\vec{\beta},\vec{X}(t)}$ is a solution of the KdV equation
$$\partial_tq+\partial_x^3q-6q\partial_xq=0$$
with $X_j(t)=X_j+4\beta_j^2t$, that is, representing $N$ solitons moving to the right at speeds $4\beta_1^2>\cdots >4\beta_N^2$. The $L^2(\R)$ norm is given by
$$\int_\R Q_{\vec{\beta},\vec{X}}(x)^2\rd x=\frac{16}3\sum_{j=1}^N\beta_j^3$$
and the spectrum of the corresponding Schrödinger operator is
\begin{equation}
 \sigma\left(-\Delta-Q_{\vec{\beta},\vec{X}}\right)=\left\{-\beta_1^2<\cdots<-\beta_N^2\right\}\cup[0,\ii).
 \label{eq:spectrum_KdV}
\end{equation}
There are exactly $N$ distinct negative eigenvalues. We call
\begin{multline}
 \cM^N:=\bigg\{-Q_{\vec{\beta},\vec{X}}\ :\quad \vec{\beta}=(\beta_1,...,\beta_N),\ \quad \beta_1>\beta_2>\cdots>\beta_N>0\\
 \sum_{j=1}^N\beta_j^3=\frac{3}{16},\quad  \vec{X}\in\R^N\bigg\}
 \label{eq:KdV-manifold}
\end{multline}
the \emph{manifold of these $L^2$--normalized KdV $N$-solitons with a reversed sign}. For $n=1$ we use the convention that
$$\cM^1=\left\{-\frac{2\left(\frac3{16}\right)^{\frac23}}{\cosh^{2}\left(\left(\frac3{16}\right)^{\frac13}x-X\right)}\ :\qquad X\in\R\right\}$$
which are the well-known $L^2$--normalized optimizers for $L^{(1)}_{3/2,1}$. Although $\cM^1$ is obviously compact modulo translations, this is not the case of $\cM^N$ for $N\geq2$. Taking for instance $\vec{X}_n=n(1,2,...,N)$ and $\vec\beta=(3/16N)^{\frac13}(1,...,1)$ one obtains a sequence of $N$-solitons behaving asymptotically like a superposition of $N$ independent $1$-solitons. It is convenient to introduce the set containing all the normalized $m$-solitons for $m\leq N$
\begin{equation}
\cM^{\leq N}:=\bigcup_{m=1}^N \cM^m.
\label{eq:closure_M_N}
\end{equation}
After taking some $\beta_j$ to 0 we see that $\cM^{\le N}$ is just the strong $L^2$--closure of $\cM^N$ (the weak closure of $\cM^N$ is different, since solitons can escape to infinity and carry some mass).

\begin{theorem}[Behavior for $\gamma=3/2$ in 1D]\label{thm:3/2}
Let $\gamma=3/2$, $d=1$ and $N\in\N$. The $L^2(\R)$--normalized maximizers for $L^{(N)}_{3/2,1}$ are exactly the functions of $\cM^{\leq N}$, that is, minus the KdV $m$-solitons with $m\leq N$. Those have exactly $m\leq N$ negative eigenvalues. The only maximizers with $N$ negative eigenvalues are the functions in $\cM^N$.
Any normalized maximizing sequence $(V_n)$ for $L^{(N)}_{3/2,1}$ as in~\eqref{eq:choice_V_n} satisfies
\begin{equation}
\lim_{n\to\ii}\rd_{L^2(\R)}\big(V_n,\cM^N\big)=\lim_{n\to\ii}\rd_{L^2(\R)}\big(V_n,\cM^{\leq N}\big)=0.
\label{eq:limit_KdV_manifold}
\end{equation}
\end{theorem}

The details of the proof are provided later in Section~\ref{sec:proof_KdV}.

\section{Main results in the critical case (CLR inequality)}\label{sec:results_critical}
In this section we present our results in the case $\gamma=0$ in dimensions $d\geq3$. Similar to the case of the Sobolev inequality, the problem has an additional invariance under dilations which has to be taken care of. We start by studying a variational problem called $\ell^{(N)}_{0,d}$ before explaining the link with the finite rank Cwikel-Lieb-Rozenblum (CLR) constant $L^{(N)}_{0,d}$.

\subsection{A variational principle for the eigenvalues of the Birman-Schwinger operator}
In this subsection we define a variational problem $\ell^{(N)}_{0,d}$, which will be useful for our study of the CLR inequality. This is related to the Yamabe problem and was studied first in this context in~\cite{AmmHum-06}. We discuss this link in Subsection~\ref{sec:Yamabe} and the implications for the CLR constant $L^{(N)}_{0,d}$ later in Subsection~\ref{sec:CLR_finite_rank}.

Let $0\leq V\in L^{d/2}(\R^d)$ and consider the Birman-Schwinger operator
$$K_V:=\frac{1}{\sqrt{-\Delta}}V\frac{1}{\sqrt{-\Delta}},$$
which is self-adjoint, non-negative and compact. In fact, $K_V=B^*B$ where $B=V^{1/2}(-\Delta)^{-1/2}$ is itself bounded by the Hardy-Littlewood-Sobolev inequality and seen to be compact by an approximation argument, see~\cite{LieSei-09,Simon-05} and~\cite[Prop.~4.18]{FraLapWei-LT}. We call $\mu_j(V)$ the max-min levels of $K_V$, which either vanish or equal the positive eigenvalues arranged in decreasing order and repeated according to their multiplicities. Note that these are the same as those of the operator
$$K'_V:=\sqrt{V}\frac{1}{-\Delta}\sqrt V$$
which is most commonly used in the literature. In terms of the above operator $B$ we have $K_V'=BB^*$.
We then define
\begin{equation}
 \boxed{\ell^{(N)}_{0,d}:=N\,\sup_{\substack{V\in L^{d/2}(\R^d)\\ V\neq0,\ V\geq0}}\frac{\mu_N(V)^{\frac{d}2}}{\int_{\R^d}V(x)^{\frac{d}2}\,\rd x}.}
 \label{eq:def_ell}
\end{equation}
Since $K_{t^2V(t\,\cdot-X)}$ is unitarily equivalent to $K_V$ for all $t>0$ and all $X\in\R^d$, the variational problem in~\eqref{eq:def_ell} is invariant under both translations and dilations. It is in fact also invariant under conformal transformations, see~\cite[Chap.~4]{LieLos-01} and~\cite{Frank-23}.
Using $K_{\lambda V}=\lambda K_V$ hence $\mu_N(\lambda V)=\lambda \mu_N(V)$, we obtain the equivalent formulas
\begin{align}
\ell^{(N)}_{0,d}=N\bigg(\inf_{\mu_N(V)\geq1}\int_{\R^d}V(x)^{\frac{d}2}\,\rd x)\bigg)^{-1}&=N\bigg(\inf_{\mu_N(V)=1}\int_{\R^d}V(x)^{\frac{d}2}\,\rd x)\bigg)^{-1}&\nn\\
&=N\,\sup_{\int_{\R^d}V^{\frac{d}2}=1}\mu_N(V)^{\frac{d}2}.\label{eq:def_ell_equivalent}
\end{align}
In the first infimum we are asking for the smallest possible $L^{d/2}(\R^d)$ norm of a potential $V\geq0$ so that $K_V$ has $N$ eigenvalues $\geq1$ or more. If $\mu_N(V)>1$, then we may replace $V$ by $(1-\eps)V$ which decreases the norm. This is how we obtain the second equality with an infimum over potentials $V$ satisfying the condition $\mu_N(V)=1$.

We now quickly mention the link with the CLR inequality. This will be discussed in more detail later in Subsection~\ref{sec:CLR_finite_rank}.  The Birman-Schwinger principle~\cite{ReeSim4,LieSei-09} states that the number of eigenvalues of $K_V$ that are $>1$ is exactly equal to the number of negative eigenvalues of $-\Delta-V$ (all counted with multiplicities). From~\eqref{eq:def_ell_equivalent} we obtain that
\begin{equation}
N\leq \ell^{(N)}_{0,d} \int_{\R^d}V(x)^{\frac{d}2}\,\rd x\qquad
\text{for all $0\leq V\in L^{\frac{d}2}(\R^d)$ with $\lambda_N(-\Delta-V)<0$.}
\label{eq:CLR_N}
\end{equation}
This gives an estimate on $N$ in terms of the size of $V$, under the assumption that $-\Delta-V$ has $N$ negative eigenvalues. The best constant in~\eqref{eq:CLR_N} is $\ell^{(N)}_{0,d}$ since we may as well replace the constraint $\mu_N(V)\geq1$ by $\mu_N(V)>1$ in~\eqref{eq:def_ell}. The finite-rank CLR inequality reads
\begin{equation}
\min\big(N,\#\{j\ :\ \lambda_j(-\Delta-V)<0\}\big)\leq L^{(N)}_{0,d} \int_{\R^d}V(x)_+^{\frac{d}2}\,\rd x
\label{eq:CLR_leqN}
\end{equation}
and provides an estimate on the number of eigenvalues truncated to $N$, in terms of the size of $V$. By looking at all the possibilities for the number of negative eigenvalues on the left, we deduce that its best constant is given by
$$\boxed{L^{(N)}_{0,d}=\max_{1\leq n\leq N}\ell^{(n)}_{0,d}.}$$
Hence studying $\ell^{(N)}_{0,d}$ will give us some information on $L^{(N)}_{0,d}$. Note that the best constant in the original CLR inequality is recovered in the limit $N\to\ii$:
$$L_{0,d}=\sup_{N\geq1} L^{(N)}_{0,d}=\lim_{N\to\ii}L^{(N)}_{0,d}=\sup_{N\geq1}\ell^{(N)}_{0,d}.$$

The variational problem $\ell^{(N)}_{0,d}$ contains more information than $L^{(N)}_{0,d}$, which only sees the largest of the $\ell^{(n)}_{0,d}$ in the window $1\leq n\leq N$. From the point of view of compactness, $\ell^{(N)}_{0,d}$ is in fact much more natural than $L^{(N)}_{0,d}$. We focus on $\ell^{(N)}_{0,d}$ in the next two subsections. In Subsection~\ref{sec:Yamabe} we state some results using the link with the Yamabe problem and in Subsection~\ref{sec:CLR_finite_rank} we finally go back to $L^{(N)}_{0,d}$.

\subsection{Bubble decomposition}
The following provides some useful elementary properties of $\ell^{(N)}_{0,d}$.

\begin{lemma}[Properties of $\ell^{(N)}_{0,d}$]\label{lem:pties_ell}
We have
\begin{equation}
\frac{N}{\ell^{(N)}_{0,d}}\leq \frac{N+1}{\ell^{(N+1)}_{0,d}}\qquad\forall N\in\N,
 \label{eq:ell_quotient_monotone}
\end{equation}
as well as
\begin{equation}
\frac{N}{\ell^{(N)}_{0,d}}\leq \frac{K}{\ell^{(K)}_{0,d}}+\frac{N-K}{\ell^{(N-K)}_{0,d}},\qquad \forall N\in\N,\ \forall K=1,...,N-1.
 \label{eq:subadditivity_ell}
\end{equation}
In particular, $\ell^{(N)}_{0,d}\geq \ell^{(1)}_{0,d}$ and
\begin{equation}
\lim_{N\to\ii} \ell^{(N)}_{0,d}=\sup_{N\geq1} \ell^{(N)}_{0,d}=L_{0,d}.
 \label{eq:limit_ell_N}
\end{equation}
\end{lemma}

\begin{proof}
The first inequality~\eqref{eq:ell_quotient_monotone} follows from $\mu_{N+1}(V)\leq \mu_N(V)$. The second~\eqref{eq:subadditivity_ell} is obtained by placing two quasi-optimizers far away with, respectively, $K$ and $N-K$ eigenvalues. This means that $N\mapsto N/\ell^{(N)}_{0,d}$ is subadditive and, by Fekete's subadditive lemma~\cite{Fekete-23}, we deduce that $(\ell^{(N)}_{0,d})^{-1}$ converges to its infimum in the limit $N\to\ii$, which is~\eqref{eq:limit_ell_N}.
\end{proof}

The inequality~\eqref{eq:subadditivity_ell} implies that $\ell^{(N)}_{0,d}$ is non-decreasing along certain sequences, such as $k\mapsto 2^kN_0$. We do not know whether it is non-decreasing over the whole set of integers.

In Section~\ref{sec:proof_bubble_critical}, we prove the following bubble decomposition which is the equivalent of Theorem~\ref{thm:bubble} in the critical case $\gamma=0$.

\begin{theorem}[Bubble decomposition -- critical case]\label{thm:bubble_critical}
Let $d\geq3$ and $N\in\N$. Let $0\leq V_n\in L^{d/2}(\R^d)$ be an optimizing sequence for $\ell^{(N)}_{0,d}$, chosen so that $\mu_N(V_n)=1$ for all $n$. Then there exists

\smallskip

\begin{itemize}[leftmargin=* ,parsep=0cm,itemsep=0cm,topsep=0cm]
 \item integers $N_1,...,N_K\geq1$ so that
\begin{equation}
\sum_{k=1}^KN_k=N\quad\text{and}\quad \frac{1}{\ell^{(N)}_{0,d}}=\sum_{k=1}^K\frac{N_k}{N}\frac{1}{\ell^{(N_k)}_{0,d}};
\label{eq:decomp_ell_N_inverse}
\end{equation}

\smallskip

\item non-negative functions $0\neq V^{(k)}\in L^{d/2}(\R^d,\R_+)$ for $k=1,...,K$, which satisfy $\mu_{N_k}(V^{(k)})=1$ and are optimizers for $\ell^{(N_k)}_{0,d}$;

\smallskip

\item sequences $x_n^{(k)}\in\R^d$ and scaling parameters $t_n^{(k)}$ with
\begin{equation}
 \liminf_{n\to\ii}\left(t_n^{(k)}t_n^{(k')}|x_n^{(k)}-x_n^{(k')}|^2+\frac{t_n^{(k)}}{t_n^{(k')}}+\frac{t_n^{(k')}}{t_n^{(k)}}\right)=+\ii
 \label{eq:orthogonal_scalings}
\end{equation}
\end{itemize}

\smallskip

\noindent such that, after extracting a (not displayed) subsequence, we have
\begin{equation}
\boxed{\lim_{n\to\ii}\norm{V_n-\sum_{k=1}^K (t_n^{(k)})^2V^{(k)}\Big(t_n^{(k)}(\cdot -x_n^{(k)})\Big)}_{L^{d/2}(\R^d)}=0.}
\label{eq:bubble_v_critical}
\end{equation}
\end{theorem}

Note that~\eqref{eq:decomp_ell_N_inverse} is somewhat different from what we found for $\gamma>0$ in Theorem~\ref{thm:bubble}. Remember, however, that $\ell^{(N)}_{0,d}$ is in general not equal to $L^{(N)}_{0,d}$. The bubble decomposition is properly stated in terms of  $\ell^{(N)}_{0,d}$ and not $L^{(N)}_{0,d}$.

The way to prove the theorem is to show that the spectrum of $K_{V_n}$ above $1$ is the approximate union of the spectra of the $K_{V^{(k)}}$. Each $V^{(k)}$ is an optimizer for $\ell^{(N_k)}_{0,d}$ and thus $K_{V^{(k)}}$ must have the eigenvalue 1.

The following is an immediate consequence of Theorem~\ref{thm:bubble_critical}.

\begin{corollary}[A compactness criterion for $\ell^{(N)}_{0,d}$]\label{cor:existence_binding_critical}
Let $N\geq2$ in dimension $d\geq3$. If
\begin{equation}
\frac{N}{\ell^{(N)}_{0,d}}< \frac{K}{\ell^{(K)}_{0,d}}+\frac{N-K}{\ell^{(N-K)}_{0,d}},\qquad \forall K=1,...,N-1,
 \label{eq:binding_critical}
\end{equation}
then all the optimizing sequences for $\ell^{(N)}_{0,d}$ are compact, up to translations and dilations, and converge in $L^{d/2}(\R^d)$ to an optimizer, after extraction of a subsequence. The condition~\eqref{eq:binding_critical} holds for instance when
$$\ell^{(N)}_{0,d}>\max_{m=1,...,N-1}\ell^{(m)}_{0,d}.$$
\end{corollary}

Together with the non-strict inequality~\eqref{eq:subadditivity_ell}, the strict inequality~\eqref{eq:binding_critical} implies, in particular, that
\begin{equation}
 \frac{1}{\ell^{(N)}_{0,d}}<\sum_{k=1}^K\frac{N_k}{N}\frac{1}{\ell^{(N_k)}_{0,d}}
 \label{eq:binding_critical2}
\end{equation}
for all $\sum_{k=1}^KN_k=N$ with $K\geq2$. By Theorem~\ref{thm:bubble_critical}, this avoids the non-compactness of optimizing sequences. The last part of the statement is because
$$\frac{K}{\ell^{(K)}_{0,d}}+\frac{N-K}{\ell^{(N-K)}_{0,d}}\geq \frac{N}{\dps \max_{m=1,...,N-1}\ell^{(m)}_{0,d}}.$$
We conclude that, as soon as one $\ell^{(N)}_{0,d}$ is higher than all the previous $\ell^{(m)}_{0,d}$, a minimizer must exist. This raises the question of whether $N\mapsto\ell^{(N)}_{0,d}$ could be constantly equal to $\ell^{(1)}_{0,d}$ or not, in which case there would be no other optimizer than the $N=1$ Sobolev potential. We discuss this later in Subsection~\ref{sec:Yamabe}.

\subsection{Euler-Lagrange equation}

Next, we discuss properties of optimizers, assuming they exist. We denote by $\dot{H}^1(\R^d)$ the homogeneous Sobolev space.

\begin{theorem}[Euler-Lagrange equation -- critical case]\label{thm:equation_critical}
Let $d\geq3$. Let $N\geq1$ be so that $\ell^{(N)}_{0,d}$ admits an optimizer $V_*$, normalized in the manner $\mu_N(V_*)=1$.
Then we have
\begin{equation}
 \mu_{N+1}(V_*)<\mu_N(V_*)=1.
 \label{eq:no_unfilled_shell_critical}
\end{equation}
There exists a finite system $f_j$ of orthogonal functions in $\dot{H}^1(\R^d)$ satisfying $(-\Delta -V_*)f_j=0$ such that
\begin{equation}
V_*=\left(\frac{N}{\ell^{(N)}_{0,d}}\sum_j|f_j|^2\right)^{\frac2{d-2}},\qquad \sum_j\int_{\R^d}|\nabla f_j|^2=1.
\label{eq:equation_critical}
\end{equation}
The potential $V_*$ is positive almost everywhere and
\begin{equation}
\begin{cases}
\text{real-analytic on $\R^d$}&\text{if $d\in\{3,4\}$,}\\
\text{$C^{1,\frac13}(\R^d)$ and real-analytic on $\{V_*>0\}$}&\text{if $d=5$,}\\
\text{$C^{0,\frac4{d-2}}(\R^d)$ and real-analytic on $\{V_*>0\}$}&\text{if $d\geq6$.}
  \end{cases}
 \label{eq:regularity_CLR}
\end{equation}
At infinity we have
\begin{equation}
 \lim_{|x|\to\ii}|x|^4V_*(x)=c_*
 \label{eq:decay_CLR}
\end{equation}
for some finite constant $c_*\geq0$. At least one of the optimizers of $\ell^{(N)}_{0,d}$ satisfies $c_*>0$.
\end{theorem}

The strict inequality~\eqref{eq:no_unfilled_shell_critical} is again a ``no unfilled shell'' result in the spirit of what has been proved in Hartree-Fock theory in~\cite{BacLieLosSol-94}.

The equation~\eqref{eq:equation_critical} can be rewritten in terms of the $f_j$'s in the form
\begin{equation}
 \left(-\Delta-\left(\frac{N}{\ell^{(N)}_{0,d}}\sum_j|f_j|^2\right)^{\frac2{d-2}}\right)f_k=0.
 \label{eq:NLS_critical}
\end{equation}
This is a generalization of the Emden-Fowler equation for systems of orthogonal functions in $\dot{H}^1(\R^d)$.

\subsection{$\ell^{(N)}_{0,d}$ and the Yamabe problem}\label{sec:Yamabe}
In Theorem~\ref{thm:existence} we have proved that
$$L^{(2N)}_{\gamma,d}>L^{(N)}_{\gamma,d}\qquad\text{when $\gamma>\max(0,2-d/2)$}.$$
Does this property persist when $\gamma=0$ for $\ell^{(N)}_{0,d}$, at least in dimension $d\geq5$? A conjecture of Glaser, Grosse and Martin in~\cite{GlaGroMar-78} (recently reformulated in~\cite{Frank-23}) states that this is not the case. The authors of~\cite{GlaGroMar-78} conjectured that $N\in\N\mapsto \ell^{(N)}_{0,d}$ attains its maximum for some finite $N_c$ in all dimensions, with $N_c=1$ in dimensions $d\in\{3,4,5,6\}$.
In this spirit, we can state the following theorem which is a reformulation of results from~\cite{GlaGroMar-78,AmmHum-06} in our context.

\begin{theorem}[(Non-)monotonicity of $\ell^{(N)}_{0,d}$]\label{thm:monotony_critical}\
\begin{itemize}[leftmargin=* ,parsep=0cm,itemsep=0cm,topsep=0cm]
 \item \emph{(Case $N=2$~\cite{AmmHum-06})} We have
\begin{equation}
\ell^{(2)}_{0,d}=\ell^{(1)}_{0,d}\qquad\text{for all $d\geq3$}
\label{eq:N=2_critical}
\end{equation}
and $\ell^{(2)}_{0,d}$ admits no optimizer. All the maximizing sequences $(V_n)$ for $\ell^{(2)}_{0,d}$, normalized as $\mu_2(V_n)=1$,  behave up to a subsequence as
$$V_n(x)=\frac{d(d-2)(t_n^{(1)})^2}{\left(1+(t_n^{(1)})^2|x -x_n^{(1)}|^2\right)^2}+\frac{d(d-2)(t_n^{(2)})^2}{\left(1+(t_n^{(2)})^2|x -x_n^{(2)}|^2\right)^2}+o(1)_{L^{\frac{d}2}(\R^d)}$$
with $(t_n^{(k)},x_n^{(k)})$ as in~\eqref{eq:orthogonal_scalings}.

\smallskip

\item \emph{(Case $N=d+2$~\cite{GlaGroMar-78,AmmHum-06})} We have
\begin{equation}
 \ell^{(d+2)}_{0,d}>\ell^{(1)}_{0,d}\qquad\text{in dimensions $d\geq7$.}
 \label{eq:AH-GGM}
\end{equation}
In particular, $\ell^{(N)}_{0,d}$ admits an optimizer for at least one $3\leq N\leq d+2$.
 \end{itemize}
 \end{theorem}

The inequality~\eqref{eq:AH-GGM} is obtained by using the trial potential
$$V_L(x)=\left(L+\frac{d-2}{2}\right)\left(L+\frac{d}{2}\right)\frac{4}{(1+|x|^2)^2} \qquad\text{for $L=1$}$$
which has $\mu_2(V_L)$ of multiplicity $d+1$ and $\mu_1(V)$ of multiplicity 1, so that $\mu_{d+2}(V_L)=1$. After a stereographic projection, $V_L$ just becomes a constant potential on the unit sphere.

In~\cite{AmmHum-06} Ammann and Humbert introduced \emph{Yamabe invariants} for general manifolds. Those happen to coincide with our problem in the case of the sphere, so that Theorem~\ref{thm:monotony_critical} is contained in~\cite{AmmHum-06} after a proper reinterpretation. We quickly explain this now. In Section~\ref{sec:proof_AH} we provide a self-contained proof of Theorem~\ref{thm:monotony_critical} in $\R^d$, for the reader's convenience.

Let $\cM$ be a smooth $d$-dimensional compact manifold endowed with a metric~$g$. The conformal Laplacian is given by
$$
L_g = -\Delta_g + \frac{d-2}{4(d-1)} R_g,
$$
where $R_g$ is the scalar curvature. This operator is selfadjoint in $L^2(\cM)$ and, since $\cM$ is compact, has purely discrete spectrum. We denote by $\lambda_k(g)$ its eigenvalues in nondecreasing order and repeated according to multiplicities. Ammann and Humbert introduced in~\cite{AmmHum-06} the minimization problem
\begin{equation}
E_N := \inf_{g\in\mathcal C}\, ({\rm Vol}_g)^{\frac2d} \lambda_N(g) \,,
 \label{eq:EN_AH}
\end{equation}
where $\mathcal C$ is a conformal class of metrics on $\cM$. The number $E_N$ is called the \emph{$N$th Yamabe invariant}. The link with our problem is as follows.

\begin{lemma}[Link with $N$th Yamabe invariant]\label{lem:link_Yamabe}
When $\cM=\bS^d$ is the unit sphere and $\mathcal C$ is the conformal class of the standard metric, we have
\begin{equation}
\frac{N}{\ell^{(N)}_{0,d}}=(E_N)^{\frac{d}2}.
\label{eq:link_Yamabe}
\end{equation}
\end{lemma}

The proof uses the Birman-Schwinger principle and the fact that $\R^d$ is conformally equivalent (via stereographic projection) to $\bS^d$ with a point removed~\cite{LieLos-01}. The details are provided in Appendix~\ref{app:Yamabe}.

Ammann and Humbert prove several results on $E_N$, some of which have consequences to our problem. For instance they show a compactness criterion for $N=2$ which they can verify for certain manifolds, see~\cite[Thm.~1.4 \& 1.5]{AmmHum-06}. This is equivalent to our Corollary~\ref{cor:existence_binding_critical}. They also derive the Euler--Lagrange equation in~\cite[Thm.~1.6]{AmmHum-06}. In~\cite[Thm.~4.1]{AmmHum-06} they show that the second Yamabe invariant $E_2$ is optimal for two disjoint spheres of equal radius, which provides the first part of Theorem~\ref{thm:monotony_critical}. Finally, the second part of Theorem~\ref{thm:monotony_critical} can be found in~\cite[Prop.~7.1]{AmmHum-06} but it is in fact also in a slightly different form in~\cite{GlaGroMar-78}, and has recently been rewritten in~\cite{Frank-23}.

Ammann and Humbert explicitly mention in~\cite[Sec.~7]{AmmHum-06} that most of their results are limited to $N=2$. Our techniques used to prove Theorems~\ref{thm:bubble_critical} and~\ref{thm:equation_critical} could be useful in the case of general manifolds, but we will not go further in this direction in this article.

\subsection{Application to the CLR constant $L^{(N)}_{0,d}$}\label{sec:CLR_finite_rank}
Let us now turn our attention to the finite-rank CLR inequality. We call $\cN_-(V)$ the largest possible integer $n$ so that $\mu_n(V)\geq1$. This also equals the number of negative eigenvalues plus zero-energy resonances of $-\Delta-V$ (that is, solutions of $(-\Delta-V)f=0$ with $0\neq f\in \dot{H}^1(\R^d)\setminus L^2(\R^d)$). The finite-rank CLR inequality can be rewritten in the form
\begin{equation}
\min\big\{N,\cN_-(V)\big\}\leq L^{(N)}_{0,d} \int_{\R^d}V(x)^{\frac{d}2}\,\rd x.
\label{eq:CLR_N2}
\end{equation}
The best constant is
$$L^{(N)}_{0,d}=\max_{1\leq n\leq N}\ell^{(n)}_{0,d}$$
and it is the same if we only count negative eigenvalues. An optimizer for~\eqref{eq:CLR_N2} is by definition an optimizer for one the largest $\ell^{(n)}_{0,d}$ for $1\leq n\leq N$. Counting the zero-energy modes is mandatory when investigating the existence of optimizers. Those modes are always present for optimizers, as we have seen in Theorem~\ref{thm:equation_critical}.  The following is a reformulation of Corollary~\ref{cor:existence_binding_critical}.

\begin{corollary}[Existence of optimizers for $L^{(N)}_{0,d}$]\label{cor:existence_critical}
Let $\gamma=0$ in dimension $d\geq3$ and $N\geq1$. Let $M\leq N$ be the smallest integer so that
$$\ell^{(M)}_{0,d}=L^{(N)}_{0,d}.$$
Then all the optimizing sequences $(V_n)$ for $\ell^{(M)}_{0,d}$ are compact up to translations and dilations, and converge in $L^{d/2}(\R^d)$ towards an optimizer $V_*$ for $\ell^{(M)}_{0,d}$, hence of $L^{(N)}_{0,d}$.
\end{corollary}

\begin{proof}[Proof of Corollary~\ref{cor:existence_critical}]
By the definition of $M$ we have either $M=1$ or $\ell^{(M)}_{0,d}>\ell^{(m)}_{0,d}$ for all $m<M$. The compactness of maximizing sequences up to translations and dilations follows from Theorem~\ref{thm:bubble_critical} in the first case and Corollary~\ref{cor:existence_binding_critical} in the second case.
\end{proof}

An optimizer potential $V_*$ for the CLR inequality
\begin{equation}
\cN_-(V)\leq L_{0,d} \int_{\R^d}V(x)^{\frac{d}2}\,\rd x
\label{eq:CLR}
\end{equation}
is by definition a potential $V_*\in L^{d/2}(\R^d)$ for which there is equality in~\eqref{eq:CLR}. Since $V_*\in L^{d/2}(\R^d)$, this potential generates finitely many eigenvalues and zero-energy resonances. We thus immediately obtain the following.

\begin{corollary}[Existence of optimizers for CLR]
There exists an optimizer $V_*\in L^{d/2}(\R^d)$ for the CLR inequality~\eqref{eq:CLR} if and only if $L_{0,d}=L^{(N)}_{0,d}$ for some $N\geq1$, that is, $n\mapsto L^{(n)}_{0,d}$ is constant for $n$ large enough.
\end{corollary}

Our last result is a reformulation of Theorem~\ref{thm:monotony_critical}.

\begin{corollary}[(Non-)monotonicity]\label{cor:L2=L1} We have
$$L^{(2)}_{0,d}=L^{(1)}_{0,d}\quad \text{in all dimensions $d\geq3$}$$
and the optimizers for $L^{(2)}_{0,d}$ are exactly the Sobolev optimizers for $L^{(1)}_{0,d}$. We also have
$$L_{0,d}\geq L^{(d+2)}_{0,d}>L^{(1)}_{0,d}\quad \text{in dimensions $d\geq7$.}$$
\end{corollary}

The rest of the paper is devoted to the proof of all our results.

\section{Proof of Theorem~\ref{thm:bubble} (bubble decomposition)}\label{sec:proof_bubbles}

We provide a self-contained proof of the bubble decomposition in the subcritical case, which does not rely on any similar results in Sobolev spaces. In Section~\ref{sec:proof_bubble_critical} below, we will provide an argument relying on existing bubble decompositions for the critical case $\gamma=0$. A similar proof can be used in the subcritical case studied here.

Throughout the proof we assume that $\gamma>1/2$ in $d=1$ and $\gamma>0$ for $d\geq2$. Let $V_n\geq0$ be a normalized maximizing sequence for $L^{(N)}_{\gamma,d}$, that is,
\begin{equation}
 \int_{\R^d}V_n(x)^{\gamma+d/2}\,\rd x=1,\qquad \lim_{n\to\ii}\sum_{j=1}^N|\lambda_j(-\Delta-V_n)|^\gamma=L^{(N)}_{\gamma,d}.
 \label{eq:choice_V_n2}
\end{equation}
The $N$ min-max levels $\lambda_{j,n}:=\lambda_j(-\Delta-V_n)$ are then bounded. After extracting a subsequence, we have
$$\lambda_{j,n}\to \Lambda_j\leq0,\qquad \forall j=1,...,N.$$
Since $L^{(N)}_{\gamma,d}\geq L^{(1)}_{\gamma,d}>0$, we know that $\Lambda_{1}<0$. The eigenvalues cannot all tend to 0.

\subsubsection*{Step 1. Absence of vanishing}
First we prove that $V_n$ cannot vanish in the sense of~\cite{Lions-84}. The precise statement is that for all $r>0$ we have
\begin{equation}
\boxed{\liminf_{n\to\ii}\sup_{y\in\R^d}\int_{B_r(y)}V_n^{\gamma+\frac{d}2}>0}
\label{eq:vanishing}
\end{equation}
where $B_r(y)$ denotes the ball of radius $r$ centered at $y$. Since $B_1$ can be covered by finitely many balls of radius $r$, the property~\eqref{eq:vanishing} for all $r>0$ is equivalent to the same for $r=1$. Its proof follows immediately from the following lemma, together with the fact that $\lambda_{1,n}\nrightarrow0$, as we have just seen.

\begin{lemma}[Vanishing]
Let $V\geq0$ be such that $\int_{\R^d}V^{\gamma+\frac{d}2}\leq 1$. Then
\begin{equation}
 \big|\lambda_1(-\Delta-V)\big|\leq C\left(\sup_{y\in\R^d}\int_{B_1(y)}V^{\gamma+\frac{d}2}\right)^{\frac{2}{d+2\gamma}}
\label{eq:estimate_lambda_1_vanishing}
\end{equation}
for a constant $C$ depending only on $\gamma$ and $d$.
\end{lemma}

\begin{proof}
We have $|\lambda_1(-\Delta-V)|^\gamma\leq L^{(1)}_{\gamma,d}\int_{\R^d}V^{\gamma+\frac{d}2}\leq L^{(1)}_{\gamma,d}$ by~\eqref{eq:LT_N} with $N=1$. Since the supremum in~\eqref{eq:estimate_lambda_1_vanishing} is also bounded above by 1, we only have to prove the bound when it is small enough.

Let $0\leq \chi\in C^\ii_c(\R^d)$ be a function supported in the unit ball centered at the origin $B_1:=B_1(0)$ with $\int_{\R^d}\chi^2=1$. Denote $\chi_{R,y}(x):=R^{-d/2}\chi((x-y)/R)$ its translation by a vector $y\in\R^d$ and dilation by $R\geq1$. Then we have the continuous partition of unity
$$\int_{\R^d}\chi_{R,y}(x)^2\,\rd y=1,\qquad\forall x\in\R^d$$
as well as the IMS formula
\begin{equation}
 \int_{\R^d}\rd y\int_{\R^d}|\nabla(\chi_{R,y}u)|^2=\int_{\R^d}|\nabla u|^2+R^{-2}\int_{\R^d}|\nabla\chi|^2\int_{\R^d}|u|^2.
 \label{eq:IMS_kinetic_continuous}
\end{equation}
Next, we write for $u\in H^1(\R^d)$
\begin{align*}
\int_{\R^d}V|u|^2&=\int_{\R^d}\rd y\int_{\R^d}V|\chi_{R,y}u|^2\\
&\leq \int_{\R^d}\rd y\int_{\R^d}|\nabla(\chi_{R,y}u)|^2+(L^{(1)}_{\gamma,d})^{\frac1\gamma}\left(\int_{B_R(y)}V^{\gamma+\frac{d}2}\right)^{\frac1\gamma}\int_{\R^d}|\chi_{R,y}u|^2\\
&\leq  \int_{\R^d}|\nabla u|^2+C\left(\sup_{y\in\R^d}\norm{V}_{L^{\gamma+\frac{d}2}(B_R(y))}^{\frac{\gamma+d/2}{\gamma}}+R^{-2}\right)\int_{\R^d}|u|^2\\
&\leq  \int_{\R^d}|\nabla u|^2+C\left(R^{\frac{d}\gamma}\sup_{y\in\R^d}\norm{V}_{L^{\gamma+\frac{d}2}(B_1(y))}^{\frac{\gamma+d/2}{\gamma}}+R^{-2}\right)\int_{\R^d}|u|^2.
\end{align*}
In the second line we used~\eqref{eq:LT_N} for $N=1$ and in the third line we used~\eqref{eq:IMS_kinetic_continuous}. Finally, in the last line we used that any ball of radius $R\geq1$ can be covered by $CR^d$ balls of radius 1.
By the variational principle, we obtain
$$\lambda_1(-\Delta-V)\geq -C\left(R^{\frac{d}\gamma}\sup_{y\in\R^d}\norm{V}_{L^{\gamma+\frac{d}2}(B_1(y))}^{\frac{\gamma+d/2}{\gamma}}+R^{-2}\right).$$
The result follows after optimizing over $R\geq1$.
\end{proof}

\subsubsection*{Step 2. Extracting the first bubble}
From~\eqref{eq:vanishing} we can find a sequence $x^{(1)}_n\in \Z^d$ such that
$$\liminf_{n\to\ii}\int_{B_1(x_n^{(1)})}V_n^{\gamma+\frac{d}2}>0.$$
This leads us to consider the new potential $V_n(x_n^{(1)}+\cdot)$ which will weakly converge to the first bubble $V^{(1)}$ after extraction of a subsequence. Since our problem is invariant under space translations, we will for simplicity of notation assume that $x_n^{(1)}\equiv0$ so that
$$\liminf_{n\to\ii}\int_{B_1}V_n^{\gamma+\frac{d}2}>0.$$
We then consider the asymptotic mass
$$\boxed{\alpha^{(1)}:=\lim_{R\to\ii}\liminf_{n\to\ii}\int_{B_R}V_n^{\gamma+\frac{d}2}\in(0,1].}$$
If $\alpha^{(1)}=1$ then the sequence $V_n$ is tight in $L^{\gamma+d/2}(\R^d)$ and we can immediately go to the next step. In this step we \textbf{assume} that
$$\alpha^{(1)}<1$$
and explain how to extract from $V_n$ a tight piece of mass $\alpha^{(1)}$ and split the eigenvalue sum into two independent pieces. Using Levy concentration functions as in~\cite{Lions-84}, we can find a sequence $R_n\to\ii$ such that, after extraction of a subsequence,
\begin{multline}
\lim_{n\to\ii}\int_{B_{R_n}}V_n^{\gamma+\frac{d}2}=\alpha^{(1)},\qquad \lim_{n\to\ii}\int_{B_{2R_n}\setminus B_{R_n}}V_n^{\gamma+\frac{d}2}=0,\\ \lim_{n\to\ii}\int_{\R^d\setminus B_{2R_n}}V_n^{\gamma+\frac{d}2}=1-\alpha^{(1)}.
\label{eq:dichotomy}
\end{multline}
Then
\begin{equation}
\lim_{n\to\ii}\norm{V_n-V_n\1_{B_{R_n}}-V_n\1_{\R^d\setminus B_{2R_n}}}_{L^{\gamma+d/2}(\R^d)}=0
\label{eq:norm_decouples}
\end{equation}
and the sequence of localized potentials $(V_n\1_{B_{R_n}})$ is tight in $L^{\gamma+d/2}(\R^d)$.

Our main result in this step is that the eigenvalue sum decouples.

\begin{proposition}[Dichotomy]\label{prop:dichotomy_spectrum}
If $\alpha^{(1)}\in(0,1)$, there exists $M\in\{1,...,N-1\}$ such that
$$L^{(N)}_{\gamma,d}=L^{(M)}_{\gamma,d}=L^{(N-M)}_{\gamma,d}.$$
and
\begin{multline}
\sum_{j=1}^N|\lambda_j(-\Delta-V_n)|^\gamma=\sum_{j=1}^M\left|\lambda_j\big(-\Delta-V_n\1_{B_{R_n}}\big)\right|^\gamma\\+\sum_{j=1}^{N-M}\left|\lambda_j\big(-\Delta-V_n\1_{\R^d\setminus B_{2R_n}}\big)\right|^\gamma+o(1)_{n\to\ii},
\label{eq:dichotomy_spectrum}
\end{multline}
after extraction of a subsequence. The potentials $V_n\1_{B_{R_n}}$ and $V_n\1_{\R^d\setminus B_{2R_n}}$ are maximizing sequences for $L^{(M)}_{\gamma,d}$ and $L^{(N-M)}_{\gamma,d}$, respectively.
\end{proposition}

The proof of Proposition~\ref{prop:dichotomy_spectrum} relies on the following lemma, which is an adaptation of well-known results~\cite{Morgan-79,MorSim-80} to the case of $n$-dependent potentials and which we state as an independent result for convenience.

\begin{lemma}[Spectrum in the case of dichotomy]\label{lem:spectrum_decouples}
Assume that $\gamma>1/2$ if $d=1$ and $\gamma>0$ if $d\geq2$.  Let $(W_n)$ be a bounded sequence in $L^{\gamma+d/2}(\R^d,\R_+)$ so that there exists a sequence $R_n\to\ii$ with
\begin{equation}
 \lim_{n\to\ii}\int_{B_{2R_n}\setminus B_{R_n}}W_n^{\gamma+d/2}=0.
 \label{eq:splitting_V_n}
\end{equation}
Then, for every $N\in\N$, there exists $M\in \{0,...,N\}$ so that
\begin{multline}
\sum_{j=1}^N|\lambda_j(-\Delta-W_n)|^\gamma=\sum_{j=1}^M\left|\lambda_j\big(-\Delta-W_n\1_{B_{R_n}}\big)\right|^\gamma\\+\sum_{j=1}^{N-M}\left|\lambda_j\big(-\Delta-W_n\1_{\R^d\setminus B_{2R_n}}\big)\right|^\gamma+o(1)_{n\to\ii}
\label{eq:dichotomy_spectrum2}
\end{multline}
after extraction of a (not displayed) subsequence.
\end{lemma}

In the lemma we do not assume that $W_n$ is a maximizing sequence, nor that it is normalized. We also do not make any assumption about its mass in $B_{R_n}$ and outside of $B_{2R_n}$. We only assume that it is very small in the annulus $\{R_n\leq |x|\leq 2R_n\}$.  We postpone the proof of the lemma and first explain the

\begin{proof}[Proof of Proposition~\ref{prop:dichotomy_spectrum}]
Using the Lieb-Thirring inequality for the two potentials $V_n\1_{B_{R_n}}$ and $V_n\1_{\R^d\setminus B_{2R_n}}$, we obtain from~\eqref{eq:dichotomy_spectrum2} and~\eqref{eq:dichotomy}
\begin{align*}
L^{(N)}_{\gamma,d}&=\sum_{j=1}^{N}|\lambda_j(-\Delta-V_n)|^\gamma+o(1)\\
&\leq L^{(M)}_{\gamma,d}\alpha^{(1)}+ L^{(N-M)}_{\gamma,d}(1-\alpha^{(1)})+o(1)\\
&\leq L^{(N)}_{\gamma,d}+o(1)
\end{align*}
since $n\mapsto L^{(n)}_{\gamma,d}$ is non-decreasing.
After passing to the limit we find that there must be equality everywhere. Since $\alpha^{(1)}\in(0,1)$ this gives
$$L^{(M)}_{\gamma,d}=L^{(N-M)}_{\gamma,d}=L^{(N)}_{\gamma,d}.$$
In particular $M$ cannot be equal to 0 or $N$ since $L^{(0)}_{\gamma,d}=0$ and we know that $L^{(N)}_{\gamma,d}\geq L^{(1)}_{\gamma,d}>0$.
In addition, we see that $V_n\1_{B_{R_n}}$ and $V_n\1_{\R^d\setminus B_{2R_n}}$ are (not necessarily normalized) maximizing sequences for $L^{(M)}_{\gamma,d}$ and $L^{(N-M)}_{\gamma,d}$, respectively. This concludes the proof of Proposition~\ref{prop:dichotomy_spectrum}.
\end{proof}

We now provide the

\begin{proof}[Proof of Lemma~\ref{lem:spectrum_decouples}]
We prove that for any fixed $j$
\begin{equation}
\lambda_j(-\Delta-W_n)=\lambda_j\begin{pmatrix}
-\Delta-W_n\1_{B_{R_n}}&0\\
0&-\Delta-W_n\1_{\R^d\setminus B_{2R_n}}
\end{pmatrix}+o(1)_{n\to\ii}
\label{eq:spectrum_decouples}
\end{equation}
where the operator on the right side acts in the Hilbert space $L^2(\R^d)^2$. The spectrum of the latter is the union of the spectra of $-\Delta-W_n\1_{B_{R_n}}$ and $-\Delta-W_n\1_{B_{2R_n}^c}$, hence~\eqref{eq:spectrum_decouples} means that the negative spectrum of $-\Delta-W_n$ is asymptotically given by the union of these two spectra. The rest of the statement follows immediately and we only discuss the proof of~\eqref{eq:spectrum_decouples}.

In order to prove the lower bound in~\eqref{eq:spectrum_decouples} we localize in the corresponding regions of space. Let $\chi\in C^\ii_c(\R^d)$ be a radial function such that $0\leq \chi\leq1$, $\chi\equiv1$ on $B_1$ and $\chi\equiv0$ on $B_{5/4}^c$. Let $\zeta\in C^\ii(\R^d)$ be a radial function such that $0\leq \zeta\leq1$, $\zeta\equiv1$ on $B_2^c$ and $\zeta\equiv0$ on $B_{7/4}$. Finally, let $\xi=\sqrt{1-\chi^2-\zeta^2}$ so that we have the partition of unity $\chi^2+\xi^2+\zeta^2=1$. Let finally $\chi_n:=\chi(x/R_n)$ and a similar definition for $\zeta_n$ and $\xi_n$. The IMS localization formula tells us that
\begin{multline}
-\Delta-W_n=(1-\eps_n)\Big(\chi_n(-\Delta-W_n\1_{B_{R_n}})\chi_n+\zeta_n(-\Delta-W_n\1_{B_{2R_n}^c})\zeta_n+\xi_n(-\Delta)\xi_n\\
-|\nabla\chi_n|^2-|\nabla\zeta_n|^2-|\nabla\xi_n|^2\Big) \\
+\eps_n\Big(-\Delta-\eps_n^{-1}W_n\1_{B_{2R_n}\setminus B_{R_n}}-W_n\1_{B_{R_n}}-W_n\1_{B_{2R_n}^c}\Big)
\label{eq:IMS}
\end{multline}
for some $0<\eps_n<1$ to be chosen later.
From the Gagliardo-Nirenberg-Sobolev inequality, we have
$$\frac{-\Delta}2-\eps_n^{-1}W_n\1_{B_{2R_n}\setminus B_{R_n}}\geq -2^{\gamma+\frac{d}2-1}\eps_n^{-\gamma-d/2}L^{(1)}_{\gamma,d}\int_{B_{2R_n}\setminus B_{R_n}}W_n^{\gamma+d/2}$$
and
$$\frac{-\Delta}2-W_n\1_{B_{R_n}}-W_n\1_{B_{2R_n}^c}\geq -2^{\gamma+\frac{d}2-1}L^{(1)}_{\gamma,d}\int_{\R^d}W_n^{\gamma+d/2}.$$
Thus we get the smallest error, of order $\eps_n$, if we choose
$$\eps_n:=\left(\int_{B_{2R_n}\setminus B_{R_n}}W_n^{\gamma+d/2}\right)^{\frac{1}{\gamma+\frac{d}2}}\underset{n\to\ii}\longrightarrow 0.$$
Using $\xi_n(-\Delta)\xi_n\geq0$ we obtain the operator inequality
\begin{multline}
-\Delta-W_n\geq (1-\eps_n)\Big(\chi_n(-\Delta-W_n\1_{B_{R_n}})\chi_n+\zeta_n(-\Delta-W_n\1_{B_{2R_n}^c})\zeta_n\Big)\\-C(R_n^{-2}+\eps_n).
\label{eq:IMS_lower_bd}
\end{multline}
Introducing the partial isometry
$$\cU_n:\phi\in L^2(\R^d)\mapsto \begin{pmatrix}
\chi_n\phi\\ \zeta_n\phi\\ \xi_n\phi\end{pmatrix}\in L^2(\R^d)^3,$$
the inequality~\eqref{eq:IMS_lower_bd} can be rephrased in the form
\begin{multline*}
-\Delta-W_n\geq (1-\eps_n)\mathcal{U}_n^{-1}\begin{pmatrix}
-\Delta-W_n\1_{B_{R_n}}&&\\
&-\Delta-W_n\1_{B_{2R_n}^c}&\\
&&0
\end{pmatrix}\cU_n\\
-C(R_n^{-2}+\eps_n).
\end{multline*}
Thus we obtain from the min-max principle that
\begin{align*}
\lambda_j(-\Delta-W_n)&\geq \lambda_j\begin{pmatrix}
-\Delta-W_n\1_{B_{R_n}}&&\\
&-\Delta-W_n\1_{B_{2R_n}^c}&\\
&&0
\end{pmatrix}-C(R_n^{-2}+\eps_n)\\
&= \lambda_j\begin{pmatrix}
-\Delta-W_n\1_{B_{R_n}}&0\\
0&-\Delta-W_n\1_{B_{2R_n}^c}
\end{pmatrix}-C(R_n^{-2}+\eps_n).
\end{align*}
We have dropped the coefficient $1-\eps_n$ since the $\lambda_j$ are all non-positive.

To show the similar reverse bound we prove first that for all $j$,
\begin{equation}
\lambda_j(-\Delta-W_n\1_{B_{R_n}})_{|B_{5R_n/4}}= \lambda_j(-\Delta-W_n\1_{B_{R_n}})+o(1)_{n\to\ii}
\label{eq:compare_Dirichlet}
\end{equation}
where the notation on the left side means that we restrict the operator to the ball of radius $5R_n/4$ with Dirichlet  boundary condition. After extracting a subsequence we may assume that
$$\lim_{n\to\ii}\lambda_j(-\Delta-W_n\1_{B_{R_n}})_{|B_{5R_n/4}}=\Lambda_j^D,\qquad \lim_{n\to\ii}\lambda_j(-\Delta-W_n\1_{B_{R_n}})=\Lambda_j.$$
The bound
\begin{equation}
 \lambda_j(-\Delta-W_n\1_{B_{R_n}})\leq \lambda_j(-\Delta-W_n\1_{B_{R_n}})_{|B_{5R_n/4}}
 \label{eq:compare_Dirichlet_upper}
\end{equation}
follows from the min-max principle and implies immediately that
$$\Lambda_j\leq \Lambda_j^D.$$
We thus only have to prove the reverse inequality $\Lambda_j\geq \Lambda_j^D$. Using the space spanned by the $j$ first eigenfunctions of the Dirichlet Laplacian on the annulus $B_{5R_n/4}\setminus B_{R_n}$ as trial functions, we obtain after scaling that
$$\lambda_j(-\Delta-W_n\1_{B_{R_n}})_{|B_{5R_n/4}}\leq \frac{C_j}{R_n^2}.$$
This shows that $\Lambda_j^D\leq0$ for all $j$. In particular, the inequality $\Lambda_j\geq \Lambda_j^D$ follows whenever $\Lambda_j=0$.

Let us now consider a $\Lambda_j<0$ and prove that $\Lambda_j\geq\Lambda_j^D$. Take a sequence of $L^2$-normalized eigenfunctions $u_{j,n}$ of $-\Delta-W_n\1_{B_{R_n}}$ associated with an eigenvalue $\lambda_{j,n}\to\Lambda_j<0$. We call $\Xi_n:=\sqrt{\xi_n^2+\zeta_n^2}$ the localization function outside of $B_{5R_n/4}$. Multiplying the equation by $\Xi_n^2$ and integrating, using that $\Xi_n$ is real, we obtain
\begin{align*}
0&=\int_{\R^d}\nabla \overline{u_{j,n}}\cdot\nabla \left(\Xi_n^2u_{j,n}\right)+|\lambda_{j,n}|\int_{\R^d}\Xi_n^2|u_{j,n}|^2\\
&=\int_{\R^d}|\nabla (\Xi_nu_{j,n})|^2+|\lambda_{j,n}|\int_{\R^d}\Xi_n^2|u_{j,n}|^2-\int_{\R^d}u_{j,n}^2|\nabla \Xi_n|^2.
\end{align*}
This proves that
$$\int_{\R^d}|\nabla (\Xi_nu_{j,n})|^2+|\Lambda_{j}|\int_{\R^d}\Xi_n^2|u_{j,n}|^2\leq |\Lambda_j-\lambda_{j,n}|+\frac{C}{R_n^2}$$
and thus $\Xi_nu_{j,n}\to0$ strongly in $H^1(\R^d)$. Hence we may replace $u_{j,n}$ by the function $v_{j,n}:=(1-\Xi_n)u_{j,n}$ at the expense of small errors both in the energy and the $L^2$ norm. We use this argument for all the $j$ first eigenfunctions and obtain functions $v_{1,n},...,v_{j,n}$ which are almost orthogonal to each other. Applying then the min-max theorem on the space spanned by these functions, we obtain after passing to the limit the desired inequality $\Lambda_j\geq\Lambda_j^D$ and this concludes our proof of~\eqref{eq:compare_Dirichlet}.

The same argument shows that
$$\lambda_j(-\Delta-W_n\1_{B_{2R_n}^c})_{|B_{7R_n/4}^c}= \lambda_j(-\Delta-W_n\1_{B_{2R_n}^c})+o(1)_{n\to\ii}$$ and we then immediately deduce from the min-max principle that
\begin{align*}
\lambda_j(-\Delta-W_n)&\leq \lambda_j(-\Delta-W_n\1_{B_{R_n}\cup B_{2R_n}^c})\\
&\leq \lambda_j(-\Delta-W_n\1_{B_{R_n}\cup B_{2R_n}^c})_{|B_{5R_n/4}\cup B_{7R_n/4}^c} \\
&=\lambda_j\begin{pmatrix}
(-\Delta-W_n\1_{B_{R_n}})_{|B_{5R_n/4}}&0\\
0&(-\Delta-W_n\1_{B_{2R_n}^c})_{|B_{7R_n/4}^c}
\end{pmatrix}\\
&=\lambda_j\begin{pmatrix}
-\Delta-W_n\1_{B_{R_n}}&0\\
0&-\Delta-W_n\1_{B_{2R_n}^c}
\end{pmatrix}+o(1)_{n\to\ii}.
\end{align*}
The first line follows because $0\leq W_n\1_{B_{R_n}\cup B_{2R_n}^c}\leq W_n$ and the second line follows like in~\eqref{eq:compare_Dirichlet_upper}. This concludes our proof of~\eqref{eq:spectrum_decouples}, hence of Lemma~\ref{lem:spectrum_decouples}.
\end{proof}

\subsubsection*{Step 3. Convergence in the tight case}

At this step we have two possibilities. Either $\alpha^{(1)}=1$ and then $V_n$ is tight, or $0<\alpha^{(1)}<1$ and then $V_n\1_{B_{R_n}}$ is a tight maximizing sequence for $L^{(M)}_{\gamma,d}$ for some $M\in\{1,...,N-1\}$. To cover these two cases at once, in this step we study the situation when $L^{(N)}_{\gamma,d}$ admits a tight maximizing sequence.

\begin{proposition}[Tight maximizing sequences]\label{prop:tight}
Assume that $L^{(N)}_{\gamma,d}$ admits a non-negative maximizing sequence $V_n\in L^{\gamma+d/2}(\R^d)$ with
$$\lim_{n\to\ii}\int_{\R^d}V_n(x)^{\gamma+d/2}\,\rd x=\alpha\in(0,\ii),\qquad \lim_{n\to\ii}\sum_{j=1}^N|\lambda_j(-\Delta-V_n)|^\gamma=L^{(N)}_{\gamma,d}\alpha$$
which converges weakly to some $V\in L^{\gamma+\frac{d}2}(\R^d)$ and is tight:
\begin{equation}
 \lim_{R\to\ii}\limsup_{n\to\ii}\int_{\R^d\setminus B_R}V_n^{\gamma+\frac{d}2}=0.
 \label{eq:tight}
\end{equation}
Then $V_n\to V$ strongly and $V$ is an optimizer for $L^{(N)}_{\gamma,d}$.
\end{proposition}

For the proof we show the convergence of all the eigenvalues $\lambda_j(-\Delta-V_n)$ for $j\leq N$. We state this in two separate lemmas of independent interest, the first dealing with only the upper bound and the weak limit.

\begin{lemma}[Weak upper semi-continuity of the eigenvalues]\label{lem:wusc}
Assume that $W_n$ converges weakly to some $W$ in $L^{\gamma+\frac{d}2}(\R^d)$. Then we have for all $j\geq1$
\begin{equation}
 \limsup_{n\to\ii}\lambda_j(-\Delta-W_n) \leq  \lambda_j(-\Delta-W).
 \label{eq:wusc}
\end{equation}
\end{lemma}

\begin{proof}[Proof of Lemma~\ref{lem:wusc}]
Let $u_1,...,u_j\in H^1(\R^d)$ be any set of orthonormal functions in $L^2(\R^d)$. Choose $\alpha_{k,n}$ with $\sum_{k=1}^j|\alpha_{k,n}|^2=1$ so that $v_n:=\sum_{k=1}^j\alpha_{k,n}u_k$ solves the finite-dimensional problem
$$\max_{\substack{w\in{\rm span}(u_1,...,u_j)\\ \|w\|_{L^2}=1}}\int_{\R^d}\left(|\nabla w|^2-W_n|w|^2\right).$$
After extracting a subsequence we can replace the limsup in~\eqref{eq:wusc} by a limit and assume that $\alpha_{k,n}\to\alpha_k$, so that $v_n\to v=\sum_{k=1}^j\alpha_{k}u_k$ strongly in $H^1(\R^d)$. From the weak convergence $W_n\wto W$ we conclude that
$$\lim_{n\to\ii}\int_{\R^d}\left(|\nabla v_n|^2-W_n|v_n|^2\right)=\int_{\R^d}\left(|\nabla v|^2-W|v|^2\right).$$
By the min-max principle we have
$$\lambda_j(-\Delta-W_n)\leq \int_{\R^d}\left(|\nabla v_n|^2-W_n|v_n|^2\right)$$
and after passing to the limit we obtain
\begin{align*}
\limsup_{n\to\ii}\lambda_j(-\Delta-W_n)&\leq \int_{\R^d}\left(|\nabla v|^2-W|v|^2\right)\\
&\leq \max_{\substack{w\in{\rm span}(u_1,...,u_j)\\ \|w\|_{L^2}=1}}\int_{\R^d}\left(|\nabla w|^2-W|w|^2\right).
\end{align*}
Minimizing over $u_1,...,u_j$ we conclude finally that~\eqref{eq:wusc} holds.
\end{proof}

Next, we prove the convergence of the eigenvalues under the additional assumption of tightness.

\begin{lemma}[Convergence of the eigenvalues for tight sequences]\label{lem:convergence_eigenvals}
Assume that $W_n$ converges weakly to some $W$ in $L^{\gamma+\frac{d}2}(\R^d)$ and is tight
\begin{equation}
 \lim_{R\to\ii}\limsup_{n\to\ii}\int_{\R^d\setminus B_R}W_n^{\gamma+\frac{d}2}=0.
 \label{eq:tight2}
\end{equation}
Then we have for all $j\geq1$
\begin{equation}
 \lim_{n\to\ii}\lambda_j(-\Delta-W_n) = \lambda_j(-\Delta-W).
 \label{eq:CV_eigenvalues}
\end{equation}
\end{lemma}

\begin{proof}
After extracting a subsequence we can assume, as usual, that the eigenvalues converge:
$$\lim_{n\to\ii}\lambda_{j}(-\Delta-W_n)=\Lambda_j,\qquad \forall j\leq1.$$
If $\Lambda_j=0$ then we must also have $\lambda_j(-\Delta-W)=0$ from the upper semi-continuity proved in~\eqref{eq:wusc}. Hence we only have to prove the convergence for the eigenvalues which do not approach the essential spectrum. The idea of the proof is to use this property to get some extra compactness at infinity.

Let us start with $j=1$ and assume that $\Lambda_1<0$. Fix an $0<E_0<-\Lambda_1$. For $n$ large enough we have $\lambda_{1,n}=\lambda_1(-\Delta-W_n)\leq  -E_0<0$, which is thus an eigenvalue. Let $u_{1,n}$ be a corresponding (real) eigenfunction:
\begin{equation}
\left(-\Delta -W_n\right)u_{1,n}=\lambda_{1,n}\,u_{1,n}.
\label{eq:eigenfn}
\end{equation}
Multiplying the equation by $u_{1,n}$ we deduce that $(u_{1,n})$ is bounded in $H^1(\R^d)$. Hence, after extracting a subsequence, we may assume that $u_{1,n}\wto u_1$ weakly in $H^1(\R^d)$. Next, we localize similarly as we did in the proof of Lemma~\ref{lem:spectrum_decouples}.  Let $\zeta_R:=\zeta(x/R)$ where $\zeta\in C^\ii(\R^d,[0,1])$ is so that $\zeta\equiv0$ on $B_1$ and $\zeta\equiv1$ on $\R^d\setminus B_2$. We multiply the equation by $\zeta_R^2u_{1,n}$ and integrate. This gives
$$\int_{\R^d}\nabla {u_{1,n}}\cdot \nabla (\zeta_R^2u_{1,n})-\lambda_{1,n}\int_{\R^d}\zeta_R^2|u_{1,n}|^2=\int_{\R^d}W_n\zeta_R^2|u_{1,n}|^2.$$
After integrating by parts we obtain
\begin{align*}
\int_{\R^d}\nabla {u_{1,n}}\cdot \nabla (\zeta_R^2u_{1,n})&=\int_{\R^d}\zeta_R^2|\nabla u_{1,n}|^2-\frac12\int_{\R^d}| u_{1,n}|^2(\Delta\zeta_R^2) \\
&\geq \int_{\R^d\setminus B_{2R}}|\nabla u_{1,n}|^2-\frac{C}{R^2}.
\end{align*}
Therefore, using $\lambda_{1,n}\leq -E_0$ and the fact that $u_{1,n}$ is bounded in $H^1(\R^d)$,  we have proved the inequality
$$\int_{\R^d\setminus B_{2R}}\left(|\nabla u_{1,n}|^2+E_0|u_{1,n}|^2\right)\leq C\norm{W_n}_{L^{\gamma+\frac{d}2}(\R^d\setminus B_R)}+\frac{C}{R^2}.$$
Due to the tightness~\eqref{eq:tight2} of $W_n$, this shows that $(u_{1,n})$ is tight in $H^1(\R^d)$. From Rellich's strong local convergence we deduce that $u_{1,n}\to u_1$ strongly in $L^p(\R^d)$ for all the exponents less than the critical Sobolev exponent, hence in particular for $p=2$ and $p=2(\gamma+d/2)'$. This implies that
$$\lim_{n\to\ii}\int_{\R^d}W_n|u_{1,n}|^2=\int_{\R^d}W|u_{1}|^2$$
and thus
\begin{align*}
\liminf_{n\to\ii}\lambda_{1,n}&=\liminf_{n\to\ii}\int_{\R^d}\left(|\nabla u_{1,n}|^2-W_n|u_{1,n}|^2\right)\\
&\geq \int_{\R^d}\left(|\nabla u_{1}|^2-W|u_{1}|^2\right)\geq\lambda_1(-\Delta-V)
\end{align*}
where in the last inequality we have used that $\|u_1\|_{L^2}=1$. This concludes our proof that $\lambda_{1,n}$ converges to $\lambda_1(-\Delta-W)$.

We then go on by induction, as soon as $\Lambda_j<0$. For instance, if $\Lambda_2<0$ we consider a second eigenfunction $u_{2,n}$ which we take orthogonal to $u_{1,n}$ for all $n$ (this is in fact automatic since the first eigenvalue is non-degenerate). The same arguments as before show that $u_{2,n}$ converges strongly to some $u_2$ in $L^p(\R^d)$, up to a subsequence, with
$$\liminf_{n\to\ii}\lambda_{2,n}\geq\int_{\R^d}\left(|\nabla u_2|^2-W|u_2|^2\right).$$
Due to the strong convergence of both $u_{1,n}$ and $u_{2,n}$ in $L^2(\R^d)$, we have $\pscal{u_1,u_2}=0$ and therefore the right side is larger than or equal to $\lambda_2(-\Delta-W)$, by the min-max principle.  This proves that $\lambda_{2,n}$ converges to $\lambda_2(-\Delta-W)$. The argument is similar for larger $j$'s, choosing always $u_{j,n}$ orthogonal to the previously considered functions. This concludes the proof of Lemma~\ref{lem:convergence_eigenvals}.
\end{proof}

With the two lemmas at hand we can now provide the proof of Proposition~\ref{prop:tight} on our maximizing sequence $(V_n)$.

\begin{proof}
The eigenvalues converge thanks to Lemma~\ref{lem:convergence_eigenvals}. Passing to the limit and using the Lieb-Thirring inequality, we deduce that
\begin{align*}
\alpha L^{(N)}_{\gamma,d}=\lim_{n\to\ii}\sum_{j=1}^N|\lambda_j(-\Delta-V_n)|^\gamma&=\sum_{j=1}^N|\lambda_j(-\Delta-V)|^\gamma\\
&\leq L^{(N)}_{\gamma,d}\int_{\R^d}V^{\gamma+\frac{d}2}\leq L^{(N)}_{\gamma,d}\alpha
\end{align*}
since $\int_{\R^d}V^{\gamma+\frac{d}2}\leq\alpha$ due to the weak convergence $V_n\wto V$. Thus there must be equality everywhere. From the strict positivity of $\alpha$ we obtain that $\int_{\R^d}V^{\gamma+\frac{d}2}=\alpha$ and thus that $V_n\to V$ strongly~\cite[Prop.~3.32]{Brezis-11}. We also infer that $V$ is an optimizer for $L^{(N)}_{\gamma,d}$.
\end{proof}

\subsubsection*{Step 4. Iterating the construction}
Since $(V_n)$ is bounded in $L^{\gamma+d/2}(\R^d)$ we may assume, after extraction of a subsequence, that $V_n\wto V=:V^{(1)}$  weakly. If $\alpha^{(1)}=1$, then $(V_n)$ is a tight maximizing sequence for $L^{(N)}_{\gamma,d}$ and thus must converge strongly by Proposition~\ref{prop:tight}. We thus take $K=1$ and $N_1=N$ and we have proved the theorem.

If  $\alpha^{(1)}<1$ we let $N_1:=M$, the integer given by Proposition~\ref{prop:dichotomy_spectrum}. From this proposition we know that $(V_n\1_{B_{R_n}})$ is a tight maximizing sequence for $L^{(N_1)}_{\gamma,d}$ (not necessarily normalized), which also converges weakly to $V^{(1)}=V$. Thus $V^{(1)}$ is a maximizer for $L^{(N_1)}_{\gamma,d}$ and the convergence is strong, by Proposition~\ref{prop:tight}. In~\eqref{eq:dichotomy} we can thus replace $V_n\1_{B_{R_n}}$ by $V^{(1)}$. Using $V^{(1)}$ as a trial state in $L^{(N_1+1)}_{\gamma,d}=L^{(N_1)}_{\gamma,d}=L^{(N)}_{\gamma,d}$, we obtain $\lambda_{N_1+1}(-\Delta-V^{(1)})=0$. This shows that $V^{(1)}$ is in fact an optimizer for $L^{(N)}_{\gamma,d}$ as well.

The argument then goes on by induction. We consider the new sequence $V_n':=V_n\1_{\R^d\setminus B_{2R_n}}$ which is a maximizing sequence for $L^{(N-M)}_{\gamma,d}=L^{(N)}_{\gamma,d}$ and to which we can apply the whole construction again. Either it is tight up to a translation and then we have $K=2$, or we can extract a tight piece as before and we go on. Since the corresponding number $N-M$ of eigenvalues for the part at infinity is a strictly decreasing sequence of integers, the construction must stop after at most $N$ steps. This concludes the proof of Theorem~\ref{thm:bubble}. \qed

\section{Proof of Theorem~\ref{thm:equation} (equation and consequences)}\label{sec:proof_equation}

Let $V_*\geq0$ be any normalized optimizer for $L^{(N)}_{\gamma,d}$ and let $M$ be the number of negative eigenvalues (counted with multiplicity) of $-\Delta-V_*$ (we allow for the possibility that $M=\ii$). If $M<N$ then $V$ is also a minimizer for $L^{(M)}_{\gamma,d}=L^{(N)}_{\gamma,d}$ and we automatically have the gap $\lambda_M<\lambda_{M+1}=0$. Replacing $N$ by $M$ we can thus, without loss of generality, always assume that
$$\lambda_{N}<0$$
throughout the proof. We denote by $u_j$ any associated system of $N$ orthonormal eigenfunctions and aim at showing that
\begin{equation}
V_*(x)=\left(\frac{2\gamma}{(d+2\gamma)L_{\gamma,d}^{(N)}}\sum_{j=1}^N|\lambda_j|^{\gamma-1}|u_j(x)|^2\right)^{\frac1{\gamma+\frac{d}2-1}}.
 \label{eq:equation_V_optimal2}
\end{equation}

\subsubsection*{Step 1. Euler-Lagrange equation, assuming the gap}
It is convenient to first \textbf{assume}
\begin{equation}
\lambda_N<\lambda_{N+1}
\label{eq:no-unfilled-shell_proof}
\end{equation}
and write the proof of~\eqref{eq:equation_V_optimal2} in this case. We will prove~\eqref{eq:no-unfilled-shell_proof} in Step 2. Consider the numbers
$$a'=3\lambda_1,\qquad a=2\lambda_1,\qquad b=\frac{2\lambda_{N}+\lambda_{N+1}}{3},\qquad  b'=\frac{\lambda_{N}+2\lambda_{N+1}}{3}$$
which are so that
$$a'<a<\lambda_1<\lambda_N<\lambda_{N+1}<b<b'<\lambda_{N+1}.$$
Let $f\in C^\ii_c(\R_-)$ be any function supported in $[a',b']$ such that $f(x)=(x)_-^\gamma$ on $[a,b]$. Let $\chi\in L^{\gamma+d/2}(\R^d,\R)$ be any test function. Then, by perturbation theory~\cite{Kato,ReeSim4}, we know that $-\Delta-V_*-\eps\chi$ possesses exactly $N$ negative eigenvalues in $[a,b]$, which converge to the unperturbed eigenvalues $\lambda_j$ in the limit $\eps\to0$. By definition of $f$, we have
$$\sum_{j=1}^N|\lambda_j(-\Delta-V_*-\eps\chi)|^\gamma=\tr f(-\Delta-V_*-\eps\chi)$$
for any small enough $\eps$. This function is differentiable in $\eps$ with
\begin{align*}
\tr f(-\Delta-V_*-\eps\chi) &=\tr f(-\Delta-V_*)-\eps\tr \big(f'(-\Delta-V_*)\chi\big)+o(\eps)\\
&=L^{(N)}_{\gamma,d}+\eps \gamma\sum_{j=1}^N |\lambda_j|^{\gamma-1}\int_{\R^d}|u_j|^2\chi+o(\eps).
\end{align*}
On the other hand, we have
$$\int_{\R^d}(V_*+\eps\chi)_+^{\gamma+\frac{d}2}=1+\eps\frac{d+2\gamma}2\int_{\R^d}V_*^{\gamma+\frac{d}2-1}\chi+o(\eps).$$
By optimality of $V_*$ we know that
$$\frac{\sum_{j=1}^N|\lambda_j(-\Delta-V_*-\eps\chi)|^\gamma}{\int_{\R^d}(V_*+\eps\chi)_+^{\gamma+\frac{d}2}}\leq L^{(N)}_{\gamma,d}.$$
Expanding the left side in $\eps$ we obtain
$$\eps\left(\gamma\sum_{j=1}^N |\lambda_j|^{\gamma-1}\int_{\R^d}|u_j|^2\chi-L^{(N)}_{\gamma,d}\frac{d+2\gamma}2\int_{\R^d}V_*^{\gamma+\frac{d}2-1}\chi\right)\leq 0.$$
Since this is valid for positive and negative small enough $\eps$, we obtain
$$\gamma\sum_{j=1}^N |\lambda_j|^{\gamma-1}\int_{\R^d}|u_j|^2\chi=L^{(N)}_{\gamma,d}\frac{d+2\gamma}2\int_{\R^d}V_*^{\gamma+\frac{d}2-1}\chi.$$
Varying $\chi$ we finally deduce that
$$\gamma\sum_{j=1}^N |\lambda_j|^{\gamma-1}|u_j|^2=L^{(N)}_{\gamma,d}\frac{d+2\gamma}2V_*^{\gamma+\frac{d}2-1}$$
which is Equation~\eqref{eq:equation_V_optimal2}.

\subsubsection*{Step 2. Gap}

Next, we prove the strict inequality $\lambda_N<\lambda_{N+1}$. We argue by contradiction and derive an equation for $V_*$ in the case that $\lambda_N=\lambda_{N+1}<0$. The argument is similar to the one above, but slightly more involved due to the degeneracy of $\lambda_N$. We call $m$ the multiplicity of $\lambda_N$ and $K$ the smallest integer so that $\lambda_{N-K}<\lambda_N$. We recall that $\lambda_1$ is always non-degenerate by Perron-Frobenius and that $N\geq2$. Thus there is always such a $K$. We then have
$$\lambda_{N-K}<\lambda_{N-K+1}=\cdots=\lambda_N=\cdots =\lambda_{N-K+m}<\lambda_{N-K+m+1}$$
with $N-K+m>N$, by definition of $m$ and $K$. Next, we look at the spectrum of $-\Delta-V_*-\eps\chi$ as we did before. For $\eps\to0$, this operator has $N-K$ eigenvalues converging to the $\lambda_j$ with $j\leq N-K$ and thus staying strictly below $\lambda_N$, together with $m>K$ eigenvalues converging to $\lambda_N$. Since we are looking at a linear perturbation in $\eps$, it is known~\cite{Kato} that the eigenvalues close to $\lambda_N$ form $m$ smooth (in fact, real-analytic) curves which all cross at $\eps=0$. Being ordered increasingly, the corresponding eigenvalues $\lambda_j(-\Delta-V_*-\eps\chi)$ are in general not smooth at $\epsilon=0$ but only Lipschitz-continuous. The $m$ curves behave to leading order as
$$\lambda_N-\eps \mu_j^\chi+o(\eps)$$
where $\mu_1^\chi\leq\cdots\leq \mu_m^\chi$ are the $m$ (ordered) eigenvalues of the $m\times m$ matrix
\begin{equation}
 \left(\int_{\R^d}\overline{u_j}u_k\chi\right)_{N-K+1\leq j,k\leq N-K+m}
 \label{eq:matrix_analytic_pert}
\end{equation}
where $u_{N-K+1},\dots,u_{N-K+m}$ is any basis of eigenfunctions in the eigenspace $\ker(-\Delta-V_*-\lambda_N)$. This is just the restriction of the multiplication operator $\chi$ to this eigenspace. In our case we always select the $K$ lowest eigenvalues in this set of $m$ curves. To leading order, those are given by the $K$ largest $\mu_j^\chi$ for $\eps>0$ and the $K$ smallest ones for $\eps<0$. Thus we obtain
\begin{multline}
\sum_{j=1}^N|\lambda_j(-\Delta-V_*-\eps\chi)|^\gamma=L^{(N)}_{\gamma,d}+\eps \gamma\sum_{j=1}^{m-K} |\lambda_j|^{\gamma-1}\int_{\R^d}|u_j|^2\chi\\
+\eps \gamma|\lambda_N|^{\gamma-1}\sum_{k=1}^K\mu_{L+k}^\chi+o(\eps)
\end{multline}
with $L=m-K$ for $\eps>0$ and $L=0$ for $\eps<0$.
The first order condition is
\begin{multline*}
\eps\bigg(\gamma\sum_{j=1}^{m-K} |\lambda_j|^{\gamma-1}\int_{\R^d}|u_j|^2\chi+ \gamma|\lambda_N|^{\gamma-1}\sum_{k=1}^K\mu_{L+k}^\chi\\
-L^{(N)}_{\gamma,d}\frac{d+2\gamma}2\int_{\R^d}V_*^{\gamma+\frac{d}2-1}\chi\bigg)\leq 0
\end{multline*}
and after looking at $\eps>0$ and $\eps<0$, this gives the series of inequalities
\begin{align*}
|\lambda_N|^{\gamma-1}\sum_{k=1}^K\mu_{m-K+k}^\chi&\leq L^{(N)}_{\gamma,d}\frac{d+2\gamma}{2\gamma}\int_{\R^d}V_*^{\gamma+\frac{d}2-1}\chi -\sum_{j=1}^{m-K} |\lambda_j|^{\gamma-1}\int_{\R^d}|u_j|^2\chi\\
&\leq |\lambda_N|^{\gamma-1}\sum_{k=1}^K\mu_{k}^\chi.
\end{align*}
The sum of the $K$ lowest eigenvalues in the last term is always less or equal than the sum of the $K$ largest in the first term. Thus there must be equality everywhere. Since $K<m$, this implies in particular that the matrix in~\eqref{eq:matrix_analytic_pert} is a multiple of the identity. Thus we must have
$$\int_{\R^d}u_j\overline{u_k}\chi=0,\qquad \int_{\R^d}(|u_j|^2-|u_k|^2)\chi=0,$$
for all $j\neq k$ and all $\chi\in L^{\gamma+d/2}(\R^d)$. The first condition implies that $u_j\overline{u_k}=0$, that is, $u_j$ and $u_k$ have disjoint supports. The second one implies $|u_j|^2=|u_k|^2$, which is a contradiction since the eigenfunctions $u_j$ cannot be identically equal to 0. Thus we conclude that $\lambda_N=\lambda_{N+1}$ cannot hold, as we claimed.

\subsubsection*{Step 3. Regularity}

It is useful at this point to introduce the parameter
\[
    p := \frac{\gamma + \frac{d}{2}}{\gamma + \frac{d}{2} - 1}=\left(\gamma+\frac{d}2\right)'.
\]
Since $\gamma > \max \{ 0, 1-d/2 \}$, we have $p > 1$.

\begin{lemma}[Regularity]
    Any optimizer $V_* \in L^{\gamma + \frac{d}{2}}(\R^d)$ for $L^{(N)}_{\gamma,d}$ is real analytic and tends to $0$ at infinity.
\end{lemma}

\begin{proof}
Upon dilating $V_*$ we can assume that it is normalized in $L^{\gamma+d/2}(\R^d)$ and still is an optimizer. As explained in the beginning of Section~\ref{sec:proof_equation}, by changing the value of $N$ we can assume that $\lambda_N<0$. Let $u_j$ and $\lambda_j$ denote, as before, the eigenfunctions and corresponding eigenvalues of $-\Delta-V_*$. We have
\begin{equation}
    (- \Delta - \lambda_j) u_j = - V_* u_j , \quad \text{and} \quad
    V_* = \left( C \sum_{j=1}^N | \lambda_j |^{\gamma - 1} | u_j |^2 \right)^{p-1}.
 \label{eq:SCF_V_u_proof}
\end{equation}
    First, since $V_* \in L^{\gamma + \frac{d}{2}}(\R^d)$, the second equation directly gives
    \[
        u_j \in L^{r_0}(\R^d), \quad \text{with} \quad r_0 := (d + 2 \gamma)(p-1).
    \]
    We now bootstrap to gain regularity. Assume $u_j \in L^r(\R^d)$ for some $r \ge r_0$, and all $1 \le j \le N$. We deduce that $V_* \in L^{\frac{r}{2(p-1)}}(\R^d)$, and that $V_* u_j \in L^{\frac{r}{2p-1}}(\R^d)$. By elliptic regularity, this gives $u_j \in W^{2, \frac{r}{2p-1}}$. Using the Sobolev embedding, we deduce that
    \[
        u_j \in L^r(\R^d) \implies \begin{cases}
            u_j \in L^q(\R^d) \quad \forall q \in [r, \infty), \quad \text{if} \quad r \ge \frac{d}{2}(2p-1) \\
            u_j \in L^q(\R^d) \quad \forall q \in [r, r^*], \quad \text{with} \quad r^* := r \left(\frac{d}{d(2p-1) - 2r}\right).
        \end{cases}
    \]
    The function $f(x) := x \left( \frac{d}{d(2p-1) - 2x} \right)$ has a unique fixed point $x = d(p-1)$, which is unstable. So, $r_* > r$ if and only if $r > d(p-1)$, which is the case since $\gamma > 0$ and thus $r \ge r_0 > d(p-1)$. Repeating the argument a finite number of times shows that $u_j \in L^q(\R^d)$ for all $q<\ii$. Thus, by the equation~\eqref{eq:SCF_V_u_proof}, $V_* \in L^q(\R^d)$ and therefore also $V_* u_j \in L^q(\R^d)$ for any $q<\infty$. By Calderon-Zygmund, we obtain that $u \in W^{2,q}(\R^d)$  for all $q<\infty$ and then by the Sobolev-Morrey embedding this proves that $u \in C^{1,\alpha}(\R^d)$ for any $\alpha<1$ and tend to 0 at infinity.

    By Harnack's inequality~\cite{Trudinger-73,AizSim-82,LieLos-01}, we have $u_1>0$ everywhere and thus $V_* > 0$ everywhere. It follows that the functions $u_j$ are real analytic~\cite{Morrey-58} for $j=1,...,N$, hence so is $V_*$.
\end{proof}

\subsubsection*{Step 4. Decay}
For $q > 1$, we denote by
\[
    [f]_q(x) := \left( \frac{1}{| \bS^{d-1} | } \int_{\bS^{d-1}} \left| f(| x | \omega)  \right|^q \rd \sigma(\omega) \right)^\frac{1}{q}=\left( \int_{{\rm SO}(d)} \big| f\big(\cR x\big) \big|^q\, \rd \cR \right)^{ \frac1q}
\]
the $q$-spherical average of $f$.

\begin{lemma}[Decay] \label{lem:decay_opt}
    Let $0\leq V_* \in L^{\gamma + \frac{d}{2}}(\R^d)$ be an optimial potential for $L^{(N)}_{\gamma,d}$. Let $(\lambda_j, u_j)$ be the corresponding eigenpairs, repeated according to their multiplicities, with $\lambda_1 < \lambda_2 \le \cdots \le \lambda_N < \lambda_{N+1}\leq 0$. Then
    \[
        \dfrac{1}{C} \dfrac{e^{ - \sqrt{| \lambda_j| } | x| }}{1 + | x |^{d-1}} \le [u_j]_2(x)
        \quad \text{and} \quad
        |u_j  (x)| + | \nabla u_j |(x) \le C \dfrac{e^{ - \sqrt{| \lambda_j| } | x| }}{1 + | x |^{d-1}}.
    \]
    In particular, $V_*$ is exponentially decreasing, with
\begin{equation}
       0 < V_*(x) \le C \left( \dfrac{e^{-\sqrt{| \lambda_N |} | x |}}{1 + | x |^{d-1}}  \right)^{\frac{2}{\gamma + \frac{d}{2} - 1}}.
       \label{eq:exp_decay}
\end{equation}
\end{lemma}

We do not provide the proof of the lemma, which follows exactly the arguments in~\cite[Lemma 19]{GonLewNaz-21}, see also~\cite{BarMer-77, HofHofSwe-85}.

From the exponential decay of $V_*$ at infinity, it is rather easy to conclude that there are finitely many negative eigenvalues.
This is in~\cite[Thm.~XIII.6]{ReeSim4} for $d=3$ and in~\cite[Chap.~IV, Thm.~6]{Glazman} for $d\neq2$, but the proof is similar in $d=2$. In fact the finiteness also follows from the CLR inequality in dimension $d\geq3$ or Bargmann's bounds for central potentials~\cite{Bargmann-52} in dimensions $d\geq1$. This concludes the proof of Theorem~\ref{thm:equation}.\qed

\section{Proof of Theorem~\ref{thm:existence} (binding for $\gamma>\max(0,2-d/2)$)}

Our proof is divided into two steps. The first is to show that if $L^{(N)}_{\gamma,d}$ admits an optimizer, then $L^{(2N)}_{\gamma,d}>L^{(N)}_{\gamma,d}$. The argument is similar to the one used in~\cite{GonLewNaz-21,FraGonLew-21} except that we have to work with the potentials instead of the eigenfunctions. The next step is to use the bubble decomposition of Theorem~\ref{thm:bubble} to infer that $L^{(N)}_{\gamma,d}$ always admits an optimizer. Only the first step uses the additional condition
\begin{equation}
\gamma>\max\left(0,2-\frac{d}2\right).
\label{eq:cond_gamma2}
\end{equation}

\subsubsection*{Step 1. Binding }
The following lemma is the only place where the condition~\eqref{eq:cond_gamma2} is needed.

\begin{lemma}[Binding]\label{lem:binding}
Assume that $\gamma$ and $d$ satisfy~\eqref{eq:cond_gamma2}. For all $N\in\N$, we have
$$L^{(2N)}_{\gamma,d}>L^{(N)}_{\gamma,d}.$$
\end{lemma}

\begin{proof}
Let $V$ be a non-negative optimizer for $L^{(N)}_{\gamma,d}$, normalized in $L^{\gamma+d/2}(\R^d)$. Let $M$ be the number of negative eigenvalues of $-\Delta-V$ and $N'=\min(N,M)$ as in Theorem~\ref{thm:equation}. If $N'<N$, then $V$ is also an optimizer for $L^{(N')}_{\gamma,d}=L^{(N)}_{\gamma,d}$. If we prove that $L^{(2N')}_{\gamma,d}>L^{(N')}_{\gamma,d}$ then we also get that
$$L^{(2N)}_{\gamma,d}\geq L^{(2N')}_{\gamma,d}>L^{(N')}_{\gamma,d}=L^{(N)}_{\gamma,d}.$$
Upon changing $N$ into $N'$, we may thus assume, without loss of generality, that $N'=N$, that is,
$$\lambda_N<\lambda_{N+1}\leq0.$$

The rest of the proof is very similar to~\cite[Thm.~4]{GonLewNaz-21} and \cite[Thm.~1--2]{FraGonLew-21}. The idea is to place two copies of $V$ far away and compute the exponentially small interaction due to tunneling effects, and show it is attractive under the assumption~\eqref{eq:cond_gamma2} on $\gamma$.

    Let $(\lambda_j, u_j)$ be the corresponding eigenpairs for $- \Delta - V$.     Recall that $p = \frac{\gamma + \frac{d}{2}}{\gamma + \frac{d}{2} - 1}$, so that $\gamma + \frac{d}{2} = \frac{p}{p-1}$. Our condition on $\gamma$ shows that $1 < p < 2$.  Using~\eqref{eq:equation_V_optimal}, we have
    \begin{equation} \label{eq:def:V_beta_rho}
          V = \left( \beta \rho \right)^{p-1}, \ \text{with} \
          \beta := \dfrac{2 \gamma}{L^{(N)}_{\gamma, d} (d + 2 \gamma)}
          \quad \text{and} \quad
          \rho := \sum_{j=1}^N | \lambda_j |^{\gamma - 1} | u_j |^2.
    \end{equation}
    For $R > 0$ large, we introduce the test potential
    \[
        V_R := \left( \beta \rho^{(-)} + \beta \rho^{(+)}  \right)^{p-1},
        \qquad \rho^{(\pm)}(x) := \rho\left(\cR_\pm x \pm \tfrac{R}{2}e_1\right)
    \]
    with $\cR_+,\cR_-\in \text{SO}(d)$ two rotations to be chosen later and $e_1=(1,0,...)$. Note that $V(\cR_\pm x\pm\frac{R}2e_1)$ is also an optimizer of $L^{(N)}_{\gamma,d}$, due to the rotational and translational invariance of our problem. We then use the potential $V_R$ to obtain a lower bound on $L^{(2N)}_{\gamma, d}$, that is
    \begin{equation} \label{eq:lowerBound_L2N}
        L^{(2N)}_{\gamma, d} \ge \dfrac{\sum_{j=1}^{2N} | \lambda_j ( - \Delta - V_R) |^\gamma }{\int_{\R^d} V_R^{\gamma + \frac{d}{2}}}.
    \end{equation}
    Like in~\cite{GonLewNaz-21,FraGonLew-21}, it is important that we take a linear combination of the densities and not of the potentials. The argument does not work with the trial potential $V\big(\cR_+x+\frac{R}2e_1\big)+V\big(\cR_-x-\frac{R}2e_1\big)$.

    To evaluate the leading term in $R$ for the denominator and the numerator of~\eqref{eq:lowerBound_L2N}, we introduce the following quantity
    \begin{align}
        A_R & := \int_{\R^d}  V_R^{\gamma + \frac{d}{2}} - 2 \int_{\R^d} V^{\gamma + \frac{d}{2}} \nn\\
        & = \int_{\R^d} \left\{ \left( \beta \rho^{(-)} + \beta \rho^{(+)} \right)^p -  \left( \beta \rho^{(-)} \right)^p  - \left( \beta \rho^{(+)} \right)^p\right\}.\label{eq:def_A_R}
    \end{align}
    Since $p > 1$ and $\rho>0$ everywhere, we have $A_R > 0$. We will see below that $A_R$ is the leading term in the expansion for both the numerator and the denominator, under the condition~\eqref{eq:cond_gamma2}.

    For the denominator of~\eqref{eq:lowerBound_L2N}, we simply write that
    \[
        \int_{\R^d}  V_R^{\gamma + \frac{d}{2}} = 2 \int_{\R^d} V^{\gamma + \frac{d}{2}} + A_R = 2 + A_R.
    \]
    We now focus on the numerator. Let $P$ be the orthogonal projector onto the $2N$-dimensional space spanned by the $N$ first eigenfunctions of  $-\Delta-(\beta\rho^{(\pm)})^{p-1}$, which we denote by $u_j^{(\pm)}(x)=u_j\big(\cR_\pm x\pm\frac{R}2e_1\big)$. From the min-max principle we have
    $$\lambda_j ( - \Delta - V_R)\leq \lambda_j ( P(- \Delta - V_R)P)$$
    and thus
    \[
         \sum_{j=1}^{2N} | \lambda_j ( - \Delta - V_R) |^\gamma \geq
        \sum_{j=1}^{2N} | \lambda_j \left( P ( - \Delta - V_R) P \right) |^\gamma.
    \]
    We introduce the maximum overlap between the distant functions:
    \[
        e_R := \max_{1 \le i,j \le N} \int_{\R^d} | u_i \big(\cR_-x - \tfrac{R}{2}e_1\big) | \cdot | u_j\big(\cR_+x + \tfrac{R}{2}e_1\big) | \rd x\underset{R\to\ii}{\longrightarrow}0.
    \]
    Let $G$ be the Gram matrix of the $u_j^{(\pm)}$, ordered in the following manner:
    $$\Psi=(\psi^{(+)},\psi^{(-)}):=(u_1^{(+)},...,u_N^{(+)},u_1^{(-)},...,u_N^{(-)}).$$
   Then $G=\int_{\R^d}\Psi^*\Psi$ is, for $R$ large enough, a $2N\times 2N$ positive-definite hermitian matrix of the form
    \begin{equation} \label{eq:def:Gramm}
        G  = \begin{pmatrix}
            \bI_n & E \\ E^* & \bI_n
        \end{pmatrix}
        \quad \text{with} \quad E = O(e_R).
    \end{equation}
The $2N$ new functions defined by the matrix relation
    \[
        \widetilde{\Psi} := \Psi G^{-1/2}
    \]
form an orthonormal system in $L^2(\R^d)$ and span the range of $P$. Thus the negative eigenvalues of $P(-\Delta-V_R)P$ coincide with those of the $2N\times 2N$ matrix
$$ G^{-1/2}\cH G^{-1/2}$$
with
$$\cH:=\begin{pmatrix}
H^{++}&H^{-+}\\
H^{+-}&H^{--}\\
\end{pmatrix},\qquad H^{\tau\sigma}_{jk}=\pscal{u_j^{(\tau)},(-\Delta-V_R)u_k^{(\sigma)}}.
$$
Using~\eqref{eq:def:Gramm} and the formula $(1 + x)^{-1/2} = 1 - \frac12 x + O(x^2)$ gives
    \[
        G^{-1/2} = \begin{pmatrix}
            \bI_n & 0 \\ 0 & \bI_n
        \end{pmatrix} -\frac12 \begin{pmatrix}
            0 & E \\ E^* & 0
        \end{pmatrix} + O(e_R^2) =: \bI_{2n} + \cE + O(e_R^2).
    \]
   It is convenient to also separate the diagonal and off-diagonal parts of $\cH$:
    \[
\cH= \cD + \cF,\qquad \cD=\begin{pmatrix}
H^{++}&0\\
0&H^{--}\\
\end{pmatrix},\qquad \cF=\begin{pmatrix}
0&H^{-+}\\
H^{+-}&0\\
\end{pmatrix}.
    \]
Using the nonlinear eigenvalue equation for $u_j^{(\pm)}$ and the fact that $V$ is bounded, we find
$\cF=O(e_R).$
Thus,
    \begin{align*}
        G^{-1/2}\cH G^{-1/2} & = (\bI_{2n} + \cE ) (\cD + \cF)(\bI_{2n} + \cE) + O(e_R^2) \\
             & = \cD + \left( \cE \cD + \cD \cE + \cF  \right) + O(e_R^2).
    \end{align*}
and after expanding the trace we obtain
\begin{align*}
\sum_{j=1}^{2N}|\lambda_j(P(-\Delta-V_R)P)|^\gamma&=\tr\big(-G^{-1/2}\cH G^{-1/2}\big)^\gamma \\[-0.3cm]
& =\Tr (- \cD)^\gamma   - \gamma \Tr   (-\cD)^{\gamma-1} \left( \cE \cD + \cD \cE + \cF  \right)  + O(e_R^2)\\
&=\Tr (- \cD)^\gamma  + O(e_R^2)\\
&=\Tr (- H^{++})^\gamma +\Tr (- H^{--})^\gamma + O(e_R^2).
\end{align*}
The term involving $\cE \cD + \cD \cE + \cF $ vanishes since $\cD$ is block-diagonal whereas $\cE$ and $\cF$ are off-diagonal.
Let us evaluate the matrix elements
    \[
        H^{++}_{ij} = \langle u_i^{(+)} , ( - \Delta - V_R) u_j^{(+)} \rangle.
    \]
     Setting $V^{+}(x) := \left( \beta \rho^{(+)}(\cR_+x + \tfrac{R}{2}e_1) \right)^{p-1}$, we have
    \[
        ( - \Delta - V_R) = ( - \Delta - V^{(+)}) + (V_R - V^{(+)} )
    \]
and thus obtain from the eigenvalue equation
    \[
        H^{++}_{ij} = \langle u_i^{(+)} , (- \Delta - V_R), u_j^{(+)} \rangle
         = \lambda_i  \delta_{ij} - \int_{\R^d} \overline{u_i^{(+)}} u_j^{(+)} (V_R - V^{(+)}).
    \]
    The potential $V_R - V^{(+)}$ is exponentially small around $- \tfrac{R}{2}\cR_+^{-1}e_1$, where the functions $u_i^{(+)}$ are the largest.
    Setting
    $$B_R :=\max_{1\leq i,j\leq N}  \int_{\R^d} |u_i^{(+)}|\,|u_j^{(+)}| (V_R - V^{(+)}),$$
    we obtain after expanding
    \begin{align*}
        \Tr( (-H^{++})^\gamma) & = \sum_{i=1}^N | \lambda_i |^\gamma + \gamma \int_{\R^d} \sum_{i=1}^N | \lambda_i |^{\gamma - 1} | u_i^{(+)} |^2 (V_R - V^{(+)}) + O (B_R^2) \\
        & = L^{(N)}_{\gamma, d} + \gamma \int_{\R^d} \rho^{(+)} (V_R - V^{(+)}) + O (B_R^2),
    \end{align*}
    where we recognized the expression of $\rho^{(+)}$ in the last line. Adding the similar term for $H^{--}$,  we end up with
    \begin{align*}
&\sum_{j=1}^{2N}|\lambda_j(P(-\Delta-V_R)P)|^\gamma\\& = 2 L^{(N)}_{\gamma, d} + \gamma \int_{\R^d}  V_R (\rho^{(+)} + \rho^{(-)}) - V^{(+)} \rho^{(+)} -  V^{(-)} \rho^{(-)} + O (B_R^2 + e_R^2) \\
        & = 2 L^{(N)}_{\gamma, d} + \gamma \beta^{p-1} \int_{\R^d} (\rho^{(+)} + \rho^{(-)})^p
         - (\rho^{(+)})^p - (\rho^{(-)})^p + O (B_R^2 + e_R^2) \\
        & = 2 L^{(N)}_{\gamma, d} + \dfrac{\gamma}{\beta} A_R +  O (B_R^2 + e_R^2)
    \end{align*}
    where we recall that $A_R$ is defined in~\eqref{eq:def_A_R}.
    Altogether, we proved that
    \begin{align*}
        L^{(2N)}_{\gamma, d} & \geq \dfrac{2 L^{(N)}_{\gamma, d} + \dfrac{\gamma}{\beta} A_R +  O (B_R^2 + e_R^2)}{2 + A_R} \\
        & = L^{(N)}_{\gamma, d} \left( 1 + \frac{1}{2} \left( \frac{\gamma}{\beta L^{(N)}_{\gamma, d}}  - 1 \right) A_R \right) + O(B_R^2 + e_R^2 + A_R^2).
    \end{align*}
    Together with the formula of $\beta$ in~\eqref{eq:def:V_beta_rho}, we get our final bound
    \begin{equation}
    \boxed{ L^{(2N)}_{\gamma, d} \geq L^{(N)}_{\gamma, d} \left( 1 + \frac{1}{2} \left( \frac{d}{2} + \gamma - 1 \right) A_R \right) + O\big(B_R^2 + e_R^2 + A_R^2\big).}
    \label{eq:final_estim_2N}
    \end{equation}
    The term $d/2 + \gamma - 1$ is always positive. Our assumption~\eqref{eq:cond_gamma2} on $\gamma$ is rather used to ensure that $B_R^2+e_R^2=o(A_R)$.
    The next lemma can be proved following the lines of~\cite[Lemma 21]{GonLewNaz-21}. It uses the pointwise bounds (averaged over rotations) of Lemma~\ref{lem:decay_opt}, together with the inequality $\lambda_1 < \lambda_2 \le \cdots \le \lambda_N < 0$.

    \begin{lemma}[Exponentially small corrections~\cite{GonLewNaz-21}]
        For every $R$ large enough, we have
        \[
            e_R  \le C R^{\frac{3-d}{2}} \exp \left( - \sqrt{| \lambda_N |} R \right),
         \]
         and
         \[
            B_R\le C R^{3-d} \exp \left( - 2 \sqrt{| \lambda_N | } R \right)
        \]
        uniformly in the rotations $\cR_+,\cR_-\in{\rm SO}(d)$. There exist rotations $\cR_+,\cR_-\in{\rm SO}(d)$ so that
        \[
            A_R \ge \frac{1}{C} R^{-2p(d-1)} \exp \left( - \sqrt{| \lambda_N | } p R \right).
         \]
         In particular, if $\gamma$ and $d$ satisfy~\eqref{eq:cond_gamma2},  we have $1<p<2$ and thus $e_R^2 = o(A_R)$ and $B_R  = O(A_R)$.
    \end{lemma}

    The lemma implies that the leading correction in~\eqref{eq:final_estim_2N} is $A_R$ and is positive. This concludes the proof of Lemma~\ref{lem:binding}
\end{proof}

\subsubsection*{Step 2. Compactness and existence}
Let $(V_n)$ be a normalized maximizing sequence for $L^{(N)}_{\gamma,d}$. Should it be non compact modulo translations, then by Theorem~\ref{thm:bubble} we know that there exist $K\geq2$ bubbles $V^{(k)}$ which are maximizers for
$$L^{(N_k)}_{\gamma,d}=L^{(N)}_{\gamma,d}.$$
Without loss of generality, we may assume that the $N_k$ are arranged in increasing order:
$$N_1\leq N_2\leq\cdots \leq N_K.$$
The smallest bubble satisfies
$$N=\sum_{k=1}^KN_k\geq KN_1$$
and thus $N_1\leq N/K\leq N/2$ since $K\geq2$. But then, from Lemma~\ref{lem:binding}, we deduce from the existence of an optimizer for $L^{(N_1)}_{\gamma,d}$ that
$$L^{(N)}_{\gamma,d}=L^{(N_1)}_{\gamma,d}<L^{(2N_1)}_{\gamma,d}\leq L^{(N)}_{\gamma,d},$$
a contradiction. This proves that $K=1$ and all the maximizing sequences must be compact modulo translations. This concludes the proof of Theorem~\ref{thm:existence}.\qed

\section{Proof of Theorem~\ref{thm:3/2} (integrable case $\gamma=3/2$ in $d=1$)}\label{sec:proof_KdV}

It is well known that the Lieb-Thirring inequality at $\gamma=3/2$ follows from trace formulas~\cite{GarGreeKruMiu-74,LieThi-76}. Those also provide the case of equality, using results on the inverse scattering problem~\cite{KayMos-56,DeiTru-79}. We recall the notation $Q_{\vec\beta,\vec{X}}$ in~\eqref{eq:KdV-N-soliton}.

\begin{lemma}[Lieb-Thirring optimizers are KdV solitons]\label{lem:LT=KdV}
Let $V\geq0$ be a normalized optimizer for $L^{(N)}_{3/2,1}$ such that $-\Delta-V$ has exactly $N$ negative eigenvalues. Then there exists $\vec\beta=(\beta_1,...,\beta_N)$ with $\beta_1>\cdots >\beta_N>0$ and $\sum_{j=1}^N\beta^3_j=3/16$, together with $\vec X\in \R^N$ so that
$$V=Q_{\vec\beta,\vec{X}}.$$
\end{lemma}

\begin{proof}
By Theorem~\ref{thm:equation} we know that $V$ is real-analytic and decays exponentially at infinity. For such a potential we can apply the trace formula~\cite[Chap.~5, Thm.~3.6]{Yafaev-10} and obtain
$$\sum_{j=1}^N|\lambda_j|^{\frac32}+\frac3\pi \int_\R\log(a(k))\,k^2\,\rd k=\frac3{16}\int_\R V(x)^2\,\rd x$$
where $a(k)\geq1$ is the amplitude of the Jost function of $-d^2/dx^2-V$. Thus for an optimizer we must have $a(k)\equiv1$. By~\cite{KayMos-56,DeiTru-79}, the only fastly-decaying potentials satisfying this property are the KdV $N$-solitons with $\beta_j:=|\lambda_j|^{1/2}$.
\end{proof}

Let us now consider any normalized maximizing sequence $(V_n)$ for $L^{(N)}_{3/2,1}$ and prove that it approaches the manifold $\cM^N$ of normalized KdV $N$-solitons. Let $V^{(1)},...,V^{(K)}$ be the bubbles obtained from Theorem~\ref{thm:bubble} after extraction of a subsequence. Let $x_n^{(k)}$ be the corresponding space translations. Each $V^{(k)}$ is a non-trivial optimizer for $L^{(N_k)}_{3/2,1}$, hence satisfies the properties of Theorem~\ref{thm:equation}. In particular, $V^{(k)}$ decreases exponentially fast at infinity. Since $L^{(N_k)}_{3/2,1}=L^{(N_k+1)}_{3/2,1}=\frac3{16}$, we must always have $\lambda_{N_k+1}(-\Delta-V^{(k)})=0$, even in the case that $N_k=N$. The $k$th bubble can never have more than $N_k$ negative eigenvalues. It can however have less and we denote by $N_k'$ the number of its negative eigenvalues, counted with multiplicity. From Lemma~\ref{lem:LT=KdV} and after rescaling, we know that $V^{(k)}=Q_{\alpha_k\vec{\beta_k},\vec{X_k}}$ where $\alpha_k:=\left(\int_\R V^{(k)}(x)^2\,\rd x\right)^{1/3}>0$ and $\vec\beta_k\in(0,\ii)^{N'_k}$ with $\sum_j\beta_{k,j}^3=3/16$. In addition, $\sum_k\alpha_k^3=1$.

We have shown that $V_n$ decomposes as a sum of solitons moving away from each other. Those also occur in the manifold $\cM^N$. In fact, by~\cite[Prop.~3.1]{KilVis-20_ppt}, there exists $Q_{\vec{\beta},\vec{Y_n}}$ with $\vec{\beta}=\cup\vec{\beta_k}$ so that
$$\lim_{n\to\ii}\bigg\|Q_{\vec{\beta},\vec{Y_n}}-\sum_{k=1}^K\underbrace{Q_{\alpha_k\vec{\beta_k},\vec{X_k}}}_{V^{(k)}}(\cdot-x_n^{(k)})\bigg\|_{L^2(\R)}=0$$
and thus we obtain
$$\lim_{n\to\ii}\Big\|Q_{\vec{\beta},\vec{Y_n}}-V_n\Big\|_{L^2(\R)}=0.$$
This proves that the $L^2$ distance of $V_n$ to $\cM^N$ (hence to its closure $\cM^{\leq N}$) tends to 0. We have extracted a subsequence at the beginning of the argument but, since the limit holds for any subsequence, the same property must hold for the whole initial sequence $(V_n)$. This concludes the proof of Theorem~\ref{thm:3/2}.\qed

\section{Proof of Theorem~\ref{thm:equation_critical} (equation and consequences for CLR)}\label{sec:proof_equation_critical}

\subsubsection*{Step 1. Euler-Lagrange equation}
The proof is very similar to that of the subcritical case $\gamma>0$ in Section~\ref{sec:proof_equation}, except that we argue on the Birman-Schwinger operator instead of the Schrödinger operator $-\Delta-V_*$.\footnote{Another possibility would be to pass to the sphere using a stereographic projection~\cite{LieLos-01} in order to remove the essential spectrum.} For $\chi\in L^{d/2}(\R^d)$ we use the fact that
$$\eps\mapsto \frac{\mu_N(V_*+\eps\chi)^{\frac{d}2}}{\int_{\R^d}(V_*(x)+\eps\chi(x))_+^{\frac{d}2}\,\rd x}$$
attains its maximum at $\eps=0$. Similar to the case $\gamma>0$, let $m$ be the multiplicity of $\mu_N(V_*)=1$ and $K$ be so that
$$1=\mu_N(V_*)=\mu_{N-K+1}(V_*)<\mu_{N-K}(V_*).$$
Let $v_1,...,v_m$ be an orthonormal basis of $\ker(K_{V_*}-1)$ where we recall that $K_{V_*}=(-\Delta)^{-1/2}V_*(-\Delta)^{-1/2}$. Then the functions
$$u_i:=\frac1{\sqrt{-\Delta}}v_i$$
span the space
$$\cN_*=\left\{f\in \dot{H}^1(\R^d)\ :\ (-\Delta-V)f=0\right\}$$
of zero-energy resonances/eigenfunctions and form an orthonormal system in $\dot{H}^1(\R^d)$. For $\eps$ small enough the operator $K_{V_*+\eps\chi}$ has exactly $m$ eigenvalues in the neighborhood of $1$, which form $m$ smooth curves intersecting at $\eps=0$. They behave to leading order like
$$1+\eps\mu_j^\chi+o(\eps),$$
where $\mu_j^\chi$ are the $m$ eigenvalues of the $m\times m$ matrix
$$M_{ij}=\pscal{u_i,\chi u_j}_{L^2},$$
which we order increasingly.
This is the restriction of the perturbation $(-\Delta)^{-\frac12}\chi(-\Delta)^{-\frac12}$ to the eigenspace $\ker(K_{V_*}-1)=\text{span}(v_1,...,v_m)$. To leading order, $\mu_N(V_*+\eps\chi)$ is thus given by
$$\mu_N(V_*+\eps\chi)=1+\begin{cases}
\eps \mu_{m-K+1}^\chi+o(\eps)&\text{if $\eps>0$,}\\
\eps \mu_{K}^\chi+o(\eps)&\text{if $\eps<0$,}
\end{cases}$$
and the first order condition yields
\begin{equation}
\boxed{ \mu_{m-K+1}^\chi\leq \frac{\ell^{(N)}_{0,d}}{N}\int_{\R^d}V_*^{\frac{d-2}2}\chi\leq \mu_K^\chi.}
 \label{eq:compare_eigenvalues}
\end{equation}
For the moment we will not use that this involves the two eigenvalues $\mu_{m-K+1}^\chi$ and $\mu_{K}^\chi$ and defer this discussion to the last step of the proof. We simply bound
$\mu_{m-K+1}^\chi\geq \mu_1^\chi$ and $\mu_{K}^\chi\leq \mu_{m}^\chi$
and obtain that $\ell^{(N)}_{0,d}/N\int_{\R^d}V^{\frac{d}2-1}\chi$ belongs to the range of the quadratic form $u\in\cN_*\mapsto \pscal{u,\chi u}_{L^2}$ restricted to the unit sphere of $\dot{H}^1(\R^d)$. Thus for every $\chi\in L^{d/2}(\R^d)$ we deduce that there exists a $u_\chi\in \cN_*$ with $\|\nabla u_\chi\|=1$ so that
\begin{equation}
 \int_{\R^d}\left(\frac{\ell^{(N)}_{0,d}}{N}V_*^{\frac{d-2}2}-|u_\chi|^2\right)\chi=0.
 \label{eq:non_separating_hyperplane}
\end{equation}
We claim that this implies
\begin{equation}
\frac{\ell^{(N)}_{0,d}}{N}V_*^{\frac{d-2}2}\in {\rm Conv}\big\{|u|^2,\
u\in\cN_*,\ \|\nabla u\|_{L^2(\R^d)}=1\big\}=:\cK,
\label{eq:equation_critical_convex}
\end{equation}
the convex hull of the squares of normalized zero modes. To prove this, we argue by contradiction following an argument of~\cite{Nadirashvili-96}. If $V_*^{d/2-1}\ell^{(N)}_{0,d}/N$ does \emph{not} belong to the closed convex set $\cK$ on the right of~\eqref{eq:equation_critical_convex}, by the Hahn-Banach theorem in $L^{d/(d-2)}(\R^d)$~\cite[Thm.~1.7]{Brezis-11}, there exists a strictly separating hyperplane, namely a $\chi\in L^{d/2}(\R^d)$ and an $\alpha\in\R$ such that
$$\frac{\ell^{(N)}_{0,d}}{N}\int_{\R^d}V_*^{\frac{d}2-1}\chi>\alpha,\qquad \int_{\R^d}\rho \chi<\alpha,\quad \forall \rho\in\cK.$$
Taking $\rho=|u_\chi|^2$ and subtracting the two inequalities provides a contradiction to~\eqref{eq:non_separating_hyperplane}. Thus we have shown~\eqref{eq:equation_critical_convex}. The space $\cN_*$ being finite-dimensional, $\cK$ is also included in a finite-dimensional space (e.g. generated by the $\overline{u_j}u_i$ for $1\leq i,j\leq m$). By Caratheodory's theorem~\cite{Eggleston}, there exist \emph{finitely many} zero modes $w_j\in \cN_*$ so that
$$V_*=\left(\frac{N}{\ell^{(N)}_{0,d}}\sum_j |w_j|^2\right)^{\frac2{d-2}},\qquad \sum_j\int_{\R^d}|\nabla w_j|^2=1.$$
Now we may consider the Gram matrix $G_{ij}:=\pscal{\nabla w_i,\nabla w_j}_{L^2}$ and diagonalize it in the form $G=U^*{\rm diag}(\alpha_1,...,\alpha_m)U$ with a unitary $U$.
Taking $f_j=\sum_{k}U_{jk}w_k$ (which are orthogonal in $\dot{H^1}(\R^d)$) and $g_j=f_j/\sqrt{\alpha_j}$ (for $\alpha_j\neq0$ which form a basis of $\cN_*$ orthonormal in $\dot{H}^1(\R^d)$), we obtain
\begin{multline}
V_*=\left(\frac{N}{\ell^{(N)}_{0,d}}\sum_{j=1}^m |f_j|^2\right)^{\frac2{d-2}}=\left(\frac{N}{\ell^{(N)}_{0,d}}\sum_{j=1}^m \alpha_j |g_j|^2\right)^{\frac2{d-2}},\\
\sum_{j=1}^m\alpha_j=\sum_{j=1}^m\int_{\R^d}|\nabla f_j|^2=1.
\label{eq:equation_critical_proof}
\end{multline}
In the following it will be convenient to use either the functions $f_j$ or $g_j$, depending on the context. Note that condition~\eqref{eq:compare_eigenvalues} remains valid in the new basis $g_j$, with $\mu_n^\chi$ the eigenvalues of the new $m\times m$ matrix $\pscal{g_i,\chi g_j}_{L^2}$.

\subsubsection*{Step 2. Regularity}
We start with the following lemma dealing with linear solutions.

\begin{lemma}[Linear case]\label{lem:1st_iteration_regularity}
Let $d\geq3$ and $V\in L^{d/2}(\R^d,\R)$. Any $f\in\dot{H}^1(\R^d)$ so that $(-\Delta-V)f=0$ belongs to $L^p(\R^d)$ for all $\frac{d}{d-2}< p<\ii$.
\end{lemma}

From the example of the Sobolev optimizer $f(x)=(1+|x|^2)^{-\frac{d-2}2}$ and its associated potential $V(x)=d(d-2)(1+|x|^2)^{-2}$, we know that the lower bound $p>\frac{d}{d-2}$ is optimal. In fact we will show later in Step~3 that an optimizer $V_*$ for $\ell^{(N)}_{0,d}$ has the same behavior $|x|^{-4}$ at infinity.

\begin{proof}[Proof of Lemma~\ref{lem:1st_iteration_regularity}]
This is a classical result (see, e.g.,~\cite{GilbargTrudinger,BreKat-79,BreLie-84} for similar statements) but we provide a quick outline of the proof for the convenience of the reader. Taking the real and imaginary parts, we can assume that $f$ is real-valued. If we formally multiply the equation by $|f|^{\alpha-1} f$ with $\alpha\geq1$ and integrate, we obtain
$$\frac{4}{(1+\alpha)^2}\int_{\R^d}|\nabla |f|^{\frac{1+\alpha}{2}}|^2= \int V|f|^{1+\alpha}.$$
We decompose $V=V\1(|V|>M)+V\1(|V|\leq M)$ and use Hölder's inequality for the first term and just $|V|\leq M$ for the second.
Using the Sobolev inequality for the gradient, we obtain
\begin{multline}
 \left(\frac{4}{(1+\alpha)^2}S_d^{-1}-\norm{V\1(|V|>M)}_{L^{\frac{d}2}}\right)\left(\int_{\R^d}|f|^{\frac{d(1+\alpha)}{d-2}}\right)^{\frac{d-2}{d}} \leq M \int |f|^{1+\alpha}.
 \label{eq:Moser}
\end{multline}
We take $M$ large enough so that the left side is positive. We conclude that if $f\in L^{p}(\R^d)$ then $f\in L^{q}(\R^d)$ with $q=\frac{d}{d-2}p$. Iterating this procedure starting at $2d/(d-2)$ we obtain the sequence of exponents $p_n=2(\frac{d}{d-2})^{n+1}$ which diverges to infinity since $d/(d-2)>1$. This proves that $f\in L^p(\R^d)$ for all $\frac{2d}{d-2} \leq p<\ii$. Note that at each step we need to adjust $M$ appropriately, depending on $\alpha$ and $V$. Under our sole assumption that $V\in L^{d/2}(\R^d)$, this cannot be made more quantitative.

In reality we cannot really multiply by $|f|^{\alpha-1} f$ since we do not yet know that $f\in L^{\frac{d(1+\alpha)}{d-2}}$. The solution to this technical problem is to rather multiply by an $H(f(x))$ where $H(t)$ behaves like $|t|^{\alpha-1} t$ close to the origin and grows slower at infinity. Taking a limit we obtain both that $f\in L^{\frac{d(1+\alpha)}{d-2}}$ and that~\eqref{eq:Moser} holds. We refer to~\cite[Sec.~8.5]{GilbargTrudinger} for the details.

It is also possible to go backwards. This time we estimate
\begin{multline*}
\int V|f|^{1+\alpha}\leq \norm{V\1(|V|\leq M)}_{L^{\frac{d}2}}\left(\int_{\R^d}|f|^{\frac{d(1+\alpha)}{d-2}}\right)^{\frac{d-2}{d}}\\
+\norm{V\1(|V|> M)}_{L^{\frac{2d}{d+2-\alpha(d-2)}}}\norm{f}_{L^{\frac{2d}{d-2}}}^{1+\alpha}.
\end{multline*}
We have $2d/(d+2-\alpha(d-2))\leq d/2$ for $\alpha\leq 1$ hence the norm of $V\1(|V|> M)$ on the right side is finite. We take $0<\alpha<1$ and $M$ small. In the limit $\alpha\to0^+$ we obtain that $f\in L^p(\R^d)$ for $p>\frac{d}{d-2} $.
\end{proof}

Applying Lemma~\ref{lem:1st_iteration_regularity} to the $g_j$'s in~\eqref{eq:equation_critical_proof} we obtain that
$$V_*\in L^p(\R^d),\qquad \text{for all}\quad \begin{cases}
\frac{d}{4}<p<\ii&\text{for $d\geq4$,}\\
1\leq p<\ii & \text{for $d=3$.}
\end{cases}$$
Using this information we can now deduce that the $g_j$ are in $C^{1,\alpha}(\R^d)$ for all $0\leq\alpha<1$ and tend to 0 at infinity. This follows immediately from the fact that
$$g_j=(-\Delta)^{-1}(V_*g_j)=\frac{c_d}{|x|^{d-2}}\ast (V_*g_j)$$
and that $V_*g_j\in L^p(\R^d)$ for all $1\leq p<\ii$, see~\cite[Thm.~10.2]{LieLos-01}. Then $V_*$ is at least continuous and bounded. From unique continuation the $g_j$ cannot vanish on a set of positive measure,\footnote{Otherwise, by~\cite{FigGos-92} there would be a point at which the function vanishes to infinite order, a situation which cannot happen by, e.g.,~\cite{Aronszajn-57,Hormander-83}. Using~\cite{HarSim-89} we can even deduce that $\{ f_j =0 \}$ has locally finite $(d-1)$-dimensional Hausdorff measure.} hence $V_*>0$ almost everywhere. In dimensions $d\in\{3,4\}$ the power $2/(d-2)$ in~\eqref{eq:equation_critical} is an integer so the nonlinear equations~\eqref{eq:NLS_critical} are polynomial. From usual regularity theory~\cite{Morrey-58} the functions $f_j$ (hence $V_*$ and $g_j$) are real-analytic. In dimensions $d\geq5$ the power is not an integer and the best we can say is that $V_*$ satisfies~\eqref{eq:regularity_CLR}. We only provide the proof for $d\geq 6$, the argument being similar for $d=5$. Since $g_j \in C^{1,\alpha}(\R^d)$ for any $\alpha<1$, we have $V_* \in C^{0,\beta}(\R^d)$ for every $\beta<4/(d-2)$. To prove the bound for $\beta = 4/(d-2)$, we use first the algebra property of $C^{0,\beta}(\R^d)$ spaces to deduce that $V_* g_j \in C^{0,\beta}(\R^d)$ for every $\beta<4/(d-2)$. Thus, by the equation and Schauder's theory, $g_j \in C^{2,\beta}(\R^d)$ for every $\beta<4/(d-2)$. In particular, $g_j \in C^{1,1}(\R^d)$ and this implies that $V_* \in C^{0,\beta}(\R^d)$ with $\beta=4/(d-2)$, as claimed. On the other hand, we still obtain that $V_*$ and the $g_j$'s are real-analytic on the open set $\{V_*>0\}$.

\subsubsection*{Step 3. Decay}
To prove that an optimizer $V_*(x)$ behaves like $|x|^{-4}$ at infinity, we use the following lemma, which relies on the conformal invariance of our problem (reminiscent of that of the Sobolev inequality~\cite{LieLos-01}).

\begin{lemma}\label{lem:transfo_sphere1}
For $f\in \dot H^1(\R^d)$ and $0\leq V\in L^{d/2}(\R^d)$, we define
$$
\tilde f(x) := |x|^{-d+2} \, f(|x|^{-2}x) \,,\qquad W(x) := |x|^{-4} V(|x|^{-2}x) \,,\quad \forall x\in\R^d\setminus\{0\} \,.
$$
Then we have $\tilde f\in \dot{H}^1(\R^d)$ and $W\in L^{d/2}(\R^d)$ with
$$\int_{\R^d}|\nabla\tilde f|^2\,\rd x =\int_{\R^d}|\nabla f|^2\,\rd x ,\qquad \int_{\R^d}W^{\frac{d}2}\,\rd x =\int_{\R^d}V^{\frac{d}2}\,\rd x $$
and
$$\mu_j(W)=\mu_j(V)\qquad\text{for all $j\geq1$. }$$
If in addition $-\Delta f = V f$, then $-\Delta \tilde f=W\tilde f$.
\end{lemma}

\begin{proof}
Let $f$ be any function in $\dot{H}^1(\R^d)$ and $\tilde f$ as in the statement. A simple change of variables provides
$$
\int_{\R^d} |\tilde f|^{\frac{2d}{d-2}}\,\rd x = \int_{\R^d} |f|^{\frac{2d}{d-2}}\,\rd x \,,
$$
hence $\tilde f\in L^{\frac{2d}{d-2}}(\R^d)$. On the other hand, $\tilde f$ is weakly differentiable in $\R^d\setminus\{0\}$ with
$$
\partial_k \tilde f(x) = |x|^{-d} \sum_\ell (\partial_\ell f)(|x|^{-2}x) \left( \delta_{k,\ell} - 2 \frac{x_kx_\ell}{|x|^2} \right) - (d-2) |x|^{-d} x_k f(|x|^{-2}x) \,.
$$
Thus,
\begin{multline*}
|\nabla \tilde f(x)|^2 = |x|^{-2d} |(\nabla f)(|x|^{-2}x)|^2  + (d-2)^2 |x|^{-2(d-1)}|f(|x|^{-2}x)|^2\\ + 2(d-2)|x|^{-2d} f(|x|^{-2}x) x \cdot (\nabla f)(|x|^{-2}x)
\end{multline*}
and, changing variables and integrating by parts we find,
$$
\int_{\R^d} |\nabla \tilde f|^2\,\rd x = \int_{\R^d} |\nabla f|^2\,\rd x \,.
$$
This proves that the distributional derivative $\nabla \tilde f$ on $\R^d\setminus\{0\}$ is square integrable. In dimension $d\geq3$, we have $\dot{H}^1(\R^d\setminus\{0\})=\dot{H}^1(\R^d)$ and thus conclude that $\nabla \tilde f$ is in fact the distributional derivative on the whole of $\R^d$, hence that $\tilde f\in \dot H^1(\R^d)$. We have thus proved that the map $f\in \dot{H}^1(\R^d)\mapsto \tilde f\in\dot{H}^1(\R^d)$ is an isometry. This isometry is onto since the tilde transformation is clearly invertible on $C^1_c(\R^d\setminus\{0\})$, which is a dense subspace of $\dot{H}^1(\R^d)$ in dimensions $d\geq3$.

For the potential, we have the relation
$$
\int_{\R^d} W^{\frac{d}2}\,\rd x = \int_{\R^d} V^{\frac{d}2}\,\rd x
$$
and thus obtain $W\in L^{d/2}(\R^d)$. In addition
$$\int_{\R^d}V|f|^2=\int_{\R^d}W|\tilde f|^2,\qquad \forall f\in \dot{H}^1(\R^d).$$
The equality of the quadratic forms and the  variational principle immediately give
$$\mu_j(W)=\mu_j(V),\qquad\forall j\geq1.$$
In fact, the two multiplication operators $V$ and $W$ on $\dot{H}^1(\R^d)$ are unitary equivalent through the isometry $f\mapsto \tilde f$, hence their spectra coincide.

Assume now that $f\in\dot{H}^1(\R^d)$ is so that $-\Delta f=Vf$. This is equivalent to the weak formulation
$$\int_{\R^d}\nabla \phi\cdot\nabla f=\int_{\R^d}V\phi f,\qquad \forall \phi\in \dot{H}^1(\R^d).$$
The tilde transformation being invertible, we obtain
\begin{equation}
	\label{eq:eqpunctured}
	\int_{\R^d} \nabla\phi\cdot \nabla \tilde f\,\rd x = \int_{\R^d} \tilde V \phi \tilde f\,\rd x
	\qquad\text{for all}\ \phi\in \dot{H}^1(\R^d),
\end{equation}
which is the weak formulation of the equation $-\Delta \tilde f=W \tilde f$.
\end{proof}

In the following we use conformal invariance to connect decay at infinity and regularity. This idea has been exploited, for instance, in~\cite{CafGidSpr-89}. See also~\cite{BorFra-20} for a recent adaptation to the  Dirac case.

Let $V_*$ be any optimizer for $\ell^{(N)}_{0,d}$, which satisfies the relation~\eqref{eq:equation_critical_proof} for some functions $f_j\in \dot H^1(\R^d)$ satisfying $-\Delta f_j = V_* f_j$. We apply Lemma~\ref{lem:transfo_sphere1} to $V_*$ and the $f_j$'s. This provides new functions $\tilde{f_j}$ and $W_*$ so that $-\Delta \tilde{f_j} = W_* \tilde{f_j}$. In addition, $W_*$ is given in terms of the $\tilde f_j$ by the same formula~\eqref{eq:equation_critical_proof} as $V_*$.  From the regularity proved in the previous step, we deduce that
$$W_*\in C^{0,\alpha}(\R^d),\qquad \alpha=\max\{1,4/(d-2)\}.$$
In particular,
$$
|W_*(x) - W_*(0)| \lesssim |x|^\alpha
\qquad\text{for all}\ |x|\leq 1
$$
which is the same as
$$
| V_*(y) - |y|^{-4} W_*(0)| \lesssim |y|^{-\alpha-4}
\qquad\text{for all}\ |y|\geq 1 \,.
$$
This proves that
\begin{equation}
\lim_{|y|\to\ii}|y|^4V_*(y)=W_*(0)
 \label{eq:limit_V_*}
\end{equation}
as we claimed. Note that if $d=3,4,5$, one has $W_*\in C^{1,\gamma}$ with $\gamma=\max\{1,(6-d)/(d-2)\}$ and therefore one can extract the next term in the asymptotics.

At this point we cannot exclude the situation that $W_*(0)=0$. Note, however, that $W_*$ is also an optimizer of $\ell^{(N)}_{0,d}$ since $\int_{\R^d}W_*^{d/2}\,\rd x=\int_{\R^d}V_*^{d/2}\,\rd x$ and $\mu_j(W_*)=\mu_j(V_*)$ for all $j$. Due to the continuity of $V_*$ at the origin, this new optimizer behaves at infinity like
$$\lim_{|x|\to\ii}|x|^4W_*(x)=V_*(0),$$
which is somewhat dual to~\eqref{eq:limit_V_*}. In general we also do not know if $V_*(0)>0$. We however know that there definitely exists an $a\in\R^d$ so that $V_*(a)>0$. Since our problem is invariant under space translations, we may always replace $V_*(x)$ by $V_*(x-a)$ which is also an optimizer for $\ell^{(N)}_{0,d}$, before introducing the function $W_*$. Thus we can always assume $V_*(0)>0$ and the new optimizer $W_*$ then has the claimed behavior
$$\lim_{|x|\to\ii}|x|^4W_*(x)=c_*>0.$$
Such a potential always has a compact nodal set.

\begin{remark}
Choosing $V_*$ appropriately, it is possible to provide a more precise estimate of the remainder. Since $V_*(0)>0$, $V_*$ is analytic in a neighborhood of $0$. Thus $|x|^4W_*(x)$ in fact admits a complete asymptotic expansion at infinity in terms of $|x|^{-2|\alpha|} x^\alpha$ where $\alpha$ is a multiindex. For instance,
$$
W_*(x) = \frac{c_*}{|x|^4} + \frac{d\cdot x}{|x|^6}  + O\left(\frac1{|x|^6}\right)
$$
with $c_*>0$ and some $d\in\R^d$. Choosing the origin so that $V_*$ is maximal at this point we can further assume that $\nabla V_*=0$, which implies $d=0$ in the previous expansion. In addition, we can always assume after an appropriate scaling that $V_*(0)=1$. Thus there always exists a minimizer of $\ell^{(N)}_{0,d}$ so that
$$
W_*(x) = \frac{1}{|x|^4} + O\left(\frac1{|x|^6}\right) \,.
$$
\end{remark}

\begin{remark}
An optimal potential $V_*$ such that $|x|^4V_*(x)\to0$ at infinity cannot decay faster than any power of $|x|$. Namely, we claim that there exists $M>0$ so that
\begin{equation}
 \limsup_{R\to\infty} R^{M} \int_{|x|>R} \frac{V_*(x)^{\frac{d-2}2}}{|x|^4} \rd x >0.
 \label{eq:not_too_fast}
\end{equation}
The assertion follows from the unique continuation theorem of Jerison and Kenig~\cite{JerKen-85}. Indeed, let $W_*$ be the same potential as above and $\tilde f\neq0$ any corresponding eigenfunction appearing in the equation, so that $W_*\geq C|\tilde f|^{4/(d-2)}$. By~\cite{JerKen-85} we know that there exists $M<\ii$ so that
$$\limsup_{r\to0} r^{-M} \int_{|x|<r} \frac{|\tilde f(x)|^2}{|x|^4} \rd x=\limsup_{R\to\ii} R^{M} \int_{|x|>R} \frac{|f(x)|^2}{|x|^4} \rd x >0,$$
otherwise $\tilde f$ would vanish to infinite order at the origin, hence be equal to 0 everywhere. Using then $V \geq C |f|^{4/(d-2)}$ we obtain the assertion~\eqref{eq:not_too_fast}.
\end{remark}

\subsubsection*{Step 4. Full rank.} As a final step we show that $K=m$, that is, the zero-energy modes are completely filled. Since we know that $V_*$ and the $g_j$ are continuous functions, we can take $\chi$ converging to a Dirac delta in~\eqref{eq:compare_eigenvalues} and we obtain
$$\mu_{m-K+1}\big(G(x)G(x)^*\big)\leq \sum_{j=1}^m\alpha_j|g_j(x)|^2\leq \mu_{K}\big(G(x)G(x)^*\big)$$
for all $x\in\R^d$, where
$$G(x):=\begin{pmatrix}
g_1(x)\\
\vdots\\
g_m(x)
\end{pmatrix}.$$
The $m\times m$ matrix $G(x)G(x)^*$ has rank one for all $x$, with the non-negative eigenvalue $\mu_m(G(x)G(x)^*)=\sum_{j=1}^m|g_j(x)|^2$. When $K<m$  we must therefore have $\mu_{m-K+1}\big(GG^*\big)=\mu_{K}\big(GG^*\big)=0$. This would imply $V_*\equiv0$, a contradiction. Thus we conclude that $K=m$, which is the same as saying that $\mu_{N+1}(V_*)<1=\mu_{N}(V_*)$ in~\eqref{eq:no_unfilled_shell_critical}. \qed

\section{Proof of Theorem~\ref{thm:bubble_critical} (bubble decomposition for CLR)}\label{sec:proof_bubble_critical}

Our proof relies on an existing bubble decomposition of G\'erard~\cite[Thm.~1.1]{Gerard-98} for the Sobolev inequality, which we use for the eigenfunctions.

\subsubsection*{Step 1. Extracting bubbles from the eigenfunctions}
Let $V_n\geq0$ be an optimizing sequence for $\ell^{(N)}_{0,d}$, normalized so that $\mu_N(V_n)=1$ for all $n$:
$$\lim_{n\to\ii}\int_{\R^d}V_n(x)^{\frac{d}2}\,\rd x=\frac{N}{\ell^{(N)}_{0,d}}.$$
Since $\ell^{(N)}_{0,d}$ is finite, the sequence $V_n$ is bounded in $L^{d/2}(\R^d)$. Then the Birman-Schwinger operator $K_{V_n}$ is bounded. After extracting a subsequence, we may thus assume that
$$\lim_{n\to\ii}\mu_j(V_n)=\mu_j$$
for all $j\geq1$. On the other hand, the definition~\eqref{eq:def_ell} of $\ell^{(j)}$ gives
$$\mu_j(V)\leq \frac{\big(\ell^{(j)}_{0,d}\big)^{\frac2d}\norm{V}_{L^{\frac{d}2}}}{j^{\frac2d}}\leq \frac{\big(L_{0,d}\big)^{\frac2d}\norm{V}_{L^{\frac{d}2}}}{j^{\frac2d}}$$
for all $V\in L^{d/2}(\R^d)$. This bound pertains in the limit $n\to\ii$ and we obtain $\mu_j\leq Cj^{-2/d}$.
Recall that $\mu_N=1$. Let $M\geq N$ be the largest integer so that $\mu_M=1$. Then we have
$$N\leq M\leq \frac{L_{0,d}}{\ell^{(N)}_{0,d}}N.$$

Consider now a basis of $M$ eigenfunctions $u_{m,n}$ of $K_{V_n}$ corresponding to the $M$ largest eigenvalues of $K_{V_n}$:
$$\frac{1}{\sqrt{-\Delta}}V_n\frac{1}{\sqrt{-\Delta}}u_{m,n}=\mu_m(V_n)\,u_{m,n},\quad \pscal{u_{m,n},u_{m',n}}_{L^2}=\delta_{mm'},$$
with $1\leq m,m'\leq M$. Define the corresponding zero-energy states
$$f_{m,n}:=\frac{1}{\sqrt{-\Delta}}u_{m,n}$$
which solve
\begin{equation}
V_nf_{m,n}=-\mu_m(V_n)\Delta f_{m,n}
 \label{eq:zero_mode_f_n_proof}
\end{equation}
and form an orthonormal system in $\dot{H}^1(\R^d)$. In other words, those diagonalize the multiplication operator $V_n$ in $\dot{H}^1(\R^d)$.

Next, we apply G\'erard's profile decomposition~\cite{Gerard-98} in $\dot{H}^1(\R^d,\C^M)$ to the $F_n=(f_{m,n})_{m=1}^M$. This provides us with a (finite, infinite, or empty) collection of scaling parameters $t_n^{(j)}>0$, translations $x_n^{(j)}\in\R^d$ and functions $0\neq F^{(j)}=(f_m^{(j)})_{m=1}^M\in \dot{H}^1(\R^d,\C^M)$ so that, after extraction of a (not displayed) subsequence,
$$\frac1{(t_n^{(j)})^{\frac{d-2}{2}}}f_{m,n}\left(\frac{\cdot+x_n^{(j)}}{t_n^{(j)}}\right)\wto f^{(j)}_m,\qquad \forall 1\leq m\leq M,\ \forall j\geq1$$
weakly in $\dot{H}^1(\R^d)$, strongly in $L^2_{\rm loc}(\R^d)$ and almost everywhere. For any $\eps>0$, there exists a $J\geq0$ so that
\begin{equation}
f_{m,n}=\sum_{j=1}^J\underbrace{ (t_n^{(j)})^{\frac{d-2}{2}}f_m^{(j)}\left(t_n^{(j)}\big(\cdot-x_n^{(j)}\big)\right)}_{=:f_{m,n}^{(j)}}+r_{m,n}^{(J)},\qquad \norm{r_{m,n}^{(J)}}_{L^{\frac{2d}{d-2}}}\leq \eps
\label{eq:bubble_critical_f_m}
\end{equation}
and
\begin{equation*}
1=\int_{\R^d}|\nabla f_{m,n}|^2=\sum_{j=1}^J \int_{\R^d}|\nabla f_m^{(j)}|^2+\int_{\R^d}|\nabla r_{m,n}^{(J)}|^2+o(1)_{n\to\ii}.
\end{equation*}
If there is no bubble then $f_{m,n}\to0$ strongly in $L^{\frac{2d}{d-2}}(\R^d)$ for all $m=1,...,M$. This is impossible since
$$1\leq \mu_m=\int_{\R^d}V_n|f_{m,n}|^2\leq\norm{V_n}_{L^{\frac{d}2}}\norm{f_{m,n}}_{L^{\frac{2d}{d-2}}}^2.$$
Thus we conclude that $J\geq1$ for $\eps$ small enough. If $J\geq2$, two $(t_n^{(j)},x_n^{(j)})$ and $(t_n^{(j')},x_n^{(j')})$ with $j\neq j'$ are orthogonal in the sense of~\eqref{eq:orthogonal_scalings}. After extracting a further subsequence, we may assume that
\begin{equation}
 V^{(j)}_n:=\frac1{(t_n^{(j)})^2}V_n\left(\frac{\cdot+x_n^{(j)}}{t_n^{(j)}}\right)\wto V^{(j)}
 \label{eq:V_n_j}
\end{equation}
weakly in $L^{d/2}(\R^d)$ for $j\geq1$.
Rescaling the eigenvalue equation~\eqref{eq:zero_mode_f_n_proof} and passing to the weak limit using the strong local compactness of $f_{m,n}$ gives
$$V^{(j)}f_{m}^{(j)}=-\mu_m\Delta f_{m}^{(j)}.$$
For each $j$ there is an $m$ so that $f_{m}^{(j)}\neq0$ and for this $m$ we infer
$$\int_{\R^d}V^{(j)}|f_{m}^{(j)}|^2=\mu_m\int_{\R^d}|\nabla f_{m}^{(j)}|^2>0$$
since $\mu_m\geq1$. Hence we have proved that $V^{(j)}\neq0$ for all $j$ and that $\mu_m$ is an eigenvalue of $K_{V^{(j)}}$, for any $j$ so that $f_m^{(j)}\neq0$.

This is valid for the whole (possibly infinite) sequence of bubbles. In the rest of the argument we fix $\eps$ small enough and only look at the $J$ first bubbles appearing in~\eqref{eq:bubble_critical_f_m}. We will in fact prove that $J$ is independent of $\eps$, that is, there are finitely many bubbles.

\subsubsection*{Step 2. Relation to the spectra of the $K_{V^{(j)}}$}
For every $j$, let $N_j$ denote the largest integer so that $\mu_{N_j}(V^{(j)})\geq1$. From the previous analysis we know that $N_j\geq1$ since one of the $\mu_m$'s is an eigenvalue. Consider a basis of orthonormal functions $(g^{(j)}_m)_{m=1}^{N_j}$ in $\dot{H}^1(\R^d)$ for $V^{(j)}$, with eigenvalues $\mu_m(V^{(j)})$. We look at the rescaled functions
$$g_{m,n}^{(j)}:=(t_n^{(j)})^{\frac{d-2}{2}}g^{(j)}_m\big(t_n^{(j)}(\cdot -x_n^{(j)})\big).$$
Recalling that  $V_n^{(j)}$ is the rescaled potential~\eqref{eq:V_n_j} at the scale $(t_n^{(j)},x_n^{(j)})$, we have
\begin{align*}
\int_{\R^d} V_n\overline{g_{m,n}^{(j)}}g_{m',n}^{(j)}&=\int_{\R^d} V_n^{(j)}\overline{g_{m}^{(j)}}g_{m'}^{(j)}\\
&\underset{n\to\ii}{\longrightarrow}\int_{\R^d} V^{(j)}\overline{g_{m}^{(j)}}g_{m'}^{(j)}=\mu_m(V^{(j)})\delta_{m,m'},
\end{align*}
due to the weak convergence of $V_n^{(j)}$ in~\eqref{eq:V_n_j}.
On the other hand, for $j\neq j'$
$$\left|\int_{\R^d} V_n\overline{g_{m,n}^{(j)}}g_{m',n}^{(j')}\right|\leq \norm{V_n}_{L^{d/2}}\left(\int_{\R^d}|g_{m,n}^{(j)}(x)|^{\frac{d}{d-2}}|g_{m,n}^{(j')}(x)|^{\frac{d}{d-2}}\,\rd x\right)^{\frac{d-2}{d}}.$$
The integral on the right side tends to zero due to~\eqref{eq:orthogonal_scalings}. In fact, we have
\begin{equation}
 \int_{\R^d}(t_n^{(j)})^{\frac1p}(t_n^{(j')})^{1-\frac1p}f\big(t_n^{(j)}(x-x_n^{(j)})\big)g\big(t_n^{(j')}(x-x_n^{(j')})\big)\,\rd x\underset{n\to\ii}{\longrightarrow}0
 \label{eq:orthogonal_scaling}
\end{equation}
for all $f\in L^p(\R^d)$ and $g\in L^{\frac{p}{p-1}}(\R^d)$ with $1<p<\ii$. We have thus shown that the matrix of  the restriction of $V_n$ to the space
\begin{equation}
 \cG_n:=\text{span}\{g_{m,n}^{(j)},\ j=1,...,J,\ m=1,...,N_j\}
 \label{eq:def_cG_n}
\end{equation}
converges in the limit $n\to\ii$ to the diagonal matrix with entries equal to the $\mu_m(V^{(j)})\geq1$. In particular the space $\cG_n$ has dimension $\sum_{j=1}^JN_j$ for $n$ large enough. The variational principle implies that
$$\lim_{n\to\ii}\mu_{\sum_{j=1}^J N_j}(V_n)\geq\min_{\substack{j=1,...,J\\m=1,...,N_j}}\mu_m(V^{(j)})\geq1$$
and thus
$$\sum_{j=1}^JN_j\leq M$$
due to the definition of $M$. Since $N_j\geq1$ we infer that there are finitely many bubbles:
$$J\leq M.$$
However $\eps$ can be taken as small as we like in the bubble decomposition~\eqref{eq:bubble_critical_f_m} whereas $J$ can be kept fixed up to a subsequence. We may thus assume that the remainder goes to $0$:
\begin{equation}
 \lim_{n\to\ii}\norm{r_{m,n}^{(J)}}_{L^{\frac{2d}{d-2}}}=0\qquad\text{for all $m=1,...,M$.}
 \label{eq:small-remainder}
\end{equation}

The matrix of $V_n$ in the space spanned by the eigenfunctions $f_{m,n}$ is diagonal. Since the remainders $r_{m,n}^{(J)}$ tend to zero strongly in $L^{\frac{2d}{d-2}}(\R^d)$, the matrix in the space spanned by the functions $f_{m,n}-r_{m,n}^{(J)}$ for $m=1,...,M$ tends to the same diagonal matrix. This proves that the space spanned by the $f_{m,n}-r_{m,n}^{(J)}$ has dimension $M$. But this space is in fact contained in $\cG_n$ in~\eqref{eq:def_cG_n}, since any non-zero rescaled bubble $f^{(j)}_{m,n}$ is an eigenfunction of $V_n^{(j)}$, hence a linear combination of the $g_{m,n}^{(j)}$. Thus we obtain the reverse inequality $M\leq \sum_{j=1}^JN_j$ and conclude that
\begin{equation}
 \sum_{j=1}^JN_j= M.
\label{eq:equality_M}
\end{equation}
It turns out that we have proved more, namely that the spectrum is exactly given in the limit by the union of the spectra of the $K_{V^{(j)}}$, including multiplicities. But in the following we will only use~\eqref{eq:equality_M}.

\subsubsection*{Step 3. Conclusion}
We recall that
$$\frac{N}{\ell_{0,d}^{(N)}}=\lim_{n\to\ii}\int_{\R^d}(V_n)^{\frac{d}2}\geq \sum_{j=1}^J\int_{\R^d}(V^{(j)})^{\frac{d}2}.$$
To prove the last inequality, the usual method is to introduce balls centered at each bubble, growing slightly faster than the scaling $t_n^{(j)}$, so that the restriction of $V_n$ to the ball still converges weakly to $V^{(j)}$ after scaling, and to discard the integral outside of the union of the balls. From the formula~\eqref{eq:def_ell_equivalent} of $\ell^{(N_j)}_{0,d}$ and Lemma~\ref{lem:pties_ell}, we obtain:
\begin{multline}
\frac{N}{\ell_{0,d}^{(N)}}\geq \sum_{j=1}^J\int_{\R^d}(V^{(j)})^{\frac{d}2}
\geq \sum_{j=1}^J \frac{N_j}{\ell_{0,d}^{(N_j)}}\mu_{N_j}(V^{(j)})^{\frac{d}2}\\
\geq \sum_{j=1}^J \frac{N_j}{\ell_{0,d}^{(N_j)}}\geq \frac{M}{\ell^{(M)}_{0,d}}\geq \frac{N}{\ell^{(N)}_{0,d}}.
\label{eq:argument_binding_bubbles}
\end{multline}
Thus there is equality everywhere:
$$\frac{N}{\ell_{0,d}^{(N)}}=\sum_{j=1}^J \frac{N_j}{\ell_{0,d}^{(N_j)}}=\frac{M}{\ell^{(M)}_{0,d}}.$$
We also deduce that $V^{(j)}$ is an optimizer for $\ell^{(N_j)}_{0,d}$ and that $\mu_{N_j}(V^{(j)})=1$. If $M>N$ then we may as well choose to estimate
$$\int_{\R^d}(V^{(j)})^{\frac{d}2}\geq \frac{N_j-1}{\ell_{0,d}^{(N_j-1)}}\mu_{N_j-1}(V^{(j)})^{\frac{d}2}$$
for one $j$ and obtain instead
$$\frac{M-1}{\ell^{(M-1)}_{0,d}}= \frac{N}{\ell^{(N)}_{0,d}}$$
since $N\leq M-1$. We infer that $V^{(j)}$ is also an optimizer for $(N_j-1)/\ell_{0,d}^{(N_j-1)}$ and that $\mu_{N_j-1}(V^{(j)})=1$. However, from Theorem~\ref{thm:equation_critical} we know then that $\mu_{N_j}(V^{(j)})<1$, a contradiction. We therefore conclude that $M=N$ and we have shown~\eqref{eq:decomp_ell_N_inverse}.
Finally, we also have
\begin{equation}
\lim_{n\to\ii}\int_{\R^d}(V_n)^{\frac{d}2}= \sum_{j=1}^J\int_{\R^d}(V^{(j)})^{\frac{d}2},
\label{eq:converge_V_n_mass}
\end{equation}
which implies the limit~\eqref{eq:bubble_v_critical}. \qed

\section{Proof of Theorem~\ref{thm:monotony_critical} (monotonicity of $\ell^{(N)}_{0,d}$)}\label{sec:proof_AH}

\subsubsection*{Step 1. Case $N=2$}
Let $0\leq V\in L^{d/2}(\R^d)$ with $V\neq0$ and $\mu_2(V)$ be the second eigenvalue of the Birman--Schwinger operator $K'_V=\sqrt{V}(-\Delta)^{-1}\sqrt{V}$. Let $f$ be a corresponding eigenfunction. Since $K'_V$ has a non-negative kernel, the first eigenfunction is non-negative by Perron-Frobenius (in fact positive on the support of $V$), and thus $f$ must change sign. Let us introduce the function $u:= V^{-1/2} f\in \dot H^1(\R^d)$ which satisfies
$$
-\Delta u = \mu_2(V)^{-1} V u \,.
$$
For $j=1,2$ we define $a_j>0$ such that $u_1:= a_1 V^{-1/2} f_+$ and $u_2 := a_2 V^{-1/2} f_-$ satisfy
$$
\int_{\R^d} V u_j^2 = 1 \,.
$$
Thus, with $E_\pm := \{\pm f > 0\}$ and $S_d$ the Sobolev constant, we have
\begin{align*}
	2 & = \int_{\R^d} V u_1^2 + \int_{\R^d} V u_2^2 \\
	& \leq \left( \int_{E_+} V^{\frac{d}2} \right)^{\frac2d} \left( \int_{\R^d} u_1^{\frac{2d}{d-2}} \right)^{\frac{d-2}d} + \left( \int_{E_-} V^{\frac{d}2} \right)^{\frac2d} \left( \int_{\R^d} u_2^{\frac{2d}{d-2}} \right)^{\frac{d-2}d} \\
	& \leq \frac1{S_d} \left( \left( \int_{E_+} V^{\frac{d}2} \right)^{\frac2d} \int_{\R^d} |\nabla u_1|^2 + \left( \int_{E_-} V^{\frac{d}2} \right)^{\frac2d} \int_{\R^d} |\nabla u_2|^2 \right) \\
	& = \frac1{S_d\mu_2(V)} \left( \left( \int_{E_+} V^{\frac{d}2} \right)^{\frac2d} \int_{\R^d} V u_1^2 + \left( \int_{E_-} V^{\frac{d}2} \right)^{\frac2d} \int_{\R^d} V u_2^2 \right) \\
	& = \frac1{S_d\mu_2(V)}  \left( \left( \int_{E_+} V^{\frac{d}2} \right)^{\frac2d} + \left( \int_{E_-} V^{\frac{d}2} \right)^{\frac2d} \right) \\
	& \leq \frac{2^{\frac{d-2}d}}{S_d\mu_2(V)}   \left( \int_{E_+} V^{\frac{d}2} + \int_{E_-} V^{\frac{d}2} \right)^{\frac2d}  =\frac{2^{\frac{d-2}d}}{S_d\mu_2(V)}  \|V\|_{L^{d/2}(\R^d)} \,.
\end{align*}
We have therefore shown that
$
\mu_2(V) \leq S_d^{-1} 2^{-\frac2d} \|V\|_{L^{d/2}(\R^d)} \,.
$
By definition of $\ell^{(2)}_{0,d}$, this means
$$
\ell_{0,d}^{(2)} \leq S_d^{-\frac{d}2}=\ell_{0,d}^{(1)} \,.
$$
The opposite inequality $\ell_{0,d}^{(2)}\geq \ell_{0,d}^{(1)}$ holds according to~\eqref{eq:subadditivity_ell} in Lemma~\ref{lem:pties_ell}. If there existed an optimizer $V$ for both $\ell^{(1)}_{0,d}$ and $\ell^{(2)}_{0,d}$, then there would be equality everywhere in the above series of inequalities. In particular, the corresponding functions $u_1$ and $u_2$ would be Sobolev optimizers. But then they would be positive over the whole of $\R^d$, a contradiction.
The behavior of maximizing sequences follows from Theorem~\ref{thm:bubble_critical} and the explicit form of maximizers for $\ell^{(1)}_{0,d}$.

\subsubsection*{Step 2. Case $N=d+2$}
To prove the inequality~\eqref{eq:AH-GGM} we consider the potential
$$V(x)=\left(L+\frac{d-2}{2}\right)\left(L+\frac{d}{2}\right)\frac{4}{(1+|x|^2)^2}$$
which, by~\cite[Lem.~15]{Frank-23}, satisfies $\mu_N(V)=1$ with
$$\int_{\R^d}V^{\frac{d}2}=\left(L+\frac{d-2}{2}\right)^{\frac{d}2}\left(L+\frac{d}{2}\right)^{\frac{d}2}|\bS^d|,\qquad N=\frac{2}{d!}\frac{(L+d-1)!\;(L+\frac{d}2)}{L!}.$$
Taking $L=1$ we find $N=d+2$ and
$$\ell^{(d+2)}_{0,d}\geq\frac{d+2}{\int_{\R^d}V^{\frac{d}2}}=\frac{2^d}{d^{\frac{d}2}(d+2)^{\frac{d-2}2}|\bS^d|}=\frac{(d-2)^{\frac{d}2}}{(d+2)^{\frac{d-2}2}}\ell^{(1)}_{0,d}.$$
The right side is greater than $\ell^{(1)}_{0,d}$ when $d\geq7$. Note that when the dimension $d$ increases, a similar result holds for larger values of $L$, see~\cite{Frank-23}. This concludes the proof of Theorem~\ref{thm:monotony_critical}.\qed

\appendix
\section{Proof of Lemma~\ref{lem:link_Yamabe}}\label{app:Yamabe}
In order to rewrite~\eqref{eq:EN_AH} in a Birman--Schwinger-type form and explain the link with our $\ell^{(N)}_{0,d}$, we pick in a conformal class $\mathcal C$ a metric $g_0$ so that
$$
\mathcal C = \left\{ u^{\frac4{d-2}} g_0 :\ 0 < u \in C^\infty(\cM) \right\} \,.
$$
For $g=u^{\frac4{d-2}} g_0 \in\mathcal C$, we have
$
\Vol_g = \int_\cM u^{\frac{2d}{d-2}} \,\rd\,\vol_{g_0}
$
as well as
$$
\nabla_g = u^{-\frac4{d-2}}\nabla_{g_0} \,,
\qquad
| \cdot |_g = u^{\frac4{d-2}} |\cdot|_{g_0},\qquad
R_g = \frac{4(d-1)}{d-2} u^{-\frac{d+2}{d-2}} L_{g_0} u \,.
$$
Thus, for any function $\phi$,
\begin{align*}
	\langle \phi,L_g \phi\rangle_{L^2(\rd\,\vol_g)} & = \int \left( |\nabla_g \phi|_g^2 + \frac{d-2}{4(d-1)}
	R_g \phi^2\right)\rd\,\vol_g \\
	&= \int \left( u^{-\frac4{d-2}} |\nabla_{g_0} \phi|_{g_0}^2 + u^{-\frac{d+2}{d-2}} (L_{g_0} u) \phi^2 \right) u^{\frac{2d}{d-2}} \,\rd\,\vol_{g_0} \\
	& = \int \left( u^2 |\nabla_{g_0} \phi|_{g_0}^2 + u (-\Delta_{g_0} u) \phi^2 + \frac{d-2}{4(d-1)}  R_{g_0} u^2 \phi^2 \right) \rd\,\vol_{g_0} \\
	& = \int \left( |\nabla_{g_0} (u\phi) |_{g_0}^2 + \frac{d-2}{4(d-1)} R_{g_0} u^2 \phi^2\right) \rd\,\vol_{g_0} \\
	& = \langle u\phi,L_g u\phi\rangle_{L^2(\rd\,\vol_{g_0})}.
\end{align*}
Moreover,
$$
\langle \phi, \phi\rangle_{L^2(\rd\,\vol_g)} = \int \phi^2 \,\rd\,\vol_{g} = \int \phi^2 u^{\frac{2d}{d-2}} \,\rd\,\vol_{g_0} = \langle u\phi,u^{\frac{4}{d-2}} u\phi\rangle_{L^2(\rd\,\vol_{g_0})} \,.
$$
This proves that
$$
\lambda_N(g) = \lambda_N\left( u^{-\frac2{d-2}} L_{g_0} u^{-\frac{2}{d-2}}\right) \,.
$$
In particular, if $L_{g_0}$ is positive definite (which is equivalent to $E_1>0$ and also to the Yamabe constant being positive), then
$$
\lambda_N(g) = \frac1{\mu_N\left(u^{\frac2{d-2}} L_{g_0}^{-1} u^{\frac2{d-2}}\right)},
$$
where $\mu_N$ denotes the $N$th eigenvalue in decreasing order. Consequently,
$$
E_N = \inf_u \frac{\left( \int u^{\frac{2d}{d-2}} \,\rd\,\vol_{g_0} \right)^{\frac2d}}{\mu_N\left(u^{\frac2{d-2}} L_{g_0}^{-1} u^{\frac2{d-2}}\right)} \,.
$$

We now specialize to the case where $\cM=\bS^d$ is the unit sphere and $\mathcal C$ is the conformal class of the standard metric. Using that $\R^d$ is conformally equivalent (via stereographic projection) to $\bS^d$ with a point removed~\cite{LieLos-01}, we obtain~\eqref{eq:link_Yamabe} after considering the potential $V:=u^{4/(d-2)}$.\qed


\subsection*{Acknowledgement} This project has received funding from the U.S. National Science Foundation (DMS-1363432 and DMS-1954995 of R.L.F.), from the German Research Foundation (EXC-2111-390814868 of R.L.F.), and from the European Research Council (ERC) under the European Union's Horizon 2020 research and innovation programme (MDFT 725528 of M.L.).


\printbibliography

\end{document}